\documentclass[a4paper]{amsart}
\usepackage{color}
\usepackage[latin1]{inputenc}
\usepackage[T1]{fontenc}
\usepackage{amsfonts}
\usepackage{amssymb}
\usepackage{amsmath}
\usepackage{amsthm}
\usepackage{stmaryrd}
\usepackage{enumerate}
\usepackage{multirow}
\usepackage{graphicx}

\usepackage{graphicx,type1cm,eso-pic,color}
\usepackage{pstricks}
\usepackage{float}
\newtheorem{theorem}{Theorem}[section]
\newtheorem{lemma}[theorem]{Lemma}
\newtheorem{proposition}[theorem]{Proposition}

\newtheorem{definition}[theorem]{Definition}
\newtheorem{assumption}[theorem]{Assumption}

\begin{document}
\setlength\arraycolsep{2pt}
\title{Nonparametric Estimation of the Transition Density Function for Diffusion Processes}
\author{Fabienne COMTE$^{\dag}$}
\email{fabienne.comte@u-paris.fr}
\author{Nicolas MARIE$^{\diamond}$}
\email{nmarie@parisnanterre.fr}
%\keywords{Projection least squares estimator ; Diffusion processes ; Transition density function}
\date{}
\maketitle
\noindent
$^{\dag}$Universit\'e Paris Cit\'e, CNRS, MAP5, F-75006 Paris, France.\\
$^{\diamond}$Universit\'e Paris Nanterre, CNRS, Modal'X, 92001 Nanterre, France.
%

% Abstract.

%
\begin{abstract}
We assume that we observe $N$ independent copies of a diffusion process on a time-interval $[0,2T]$. For a given time $t$, we estimate the transition density $p_t(x,y)$, namely the conditional density of $X_{t + s}$ given $X_s = x$, under conditions on the diffusion coefficients ensuring that this quantity exists. We use a least squares projection method on a product of finite dimensional spaces, prove risk bounds for the estimator and propose an anisotropic model selection method, relying on several reference norms. A simulation study illustrates the theoretical part for Ornstein-Uhlenbeck or square-root (Cox-Ingersoll-Ross) processes.
\end{abstract}
%

% Section : Introduction.

%
\section{Introduction}\label{section_introduction}
In this paper, we consider the stochastic differential equation (SDE)
\begin{equation}\label{main_equation}
X_t = x_0 +\int_{0}^{t}b(X_s)ds +
\int_{0}^{t}\sigma(X_s)dW_s
\textrm{ $;$ }
t\in [0,2T],
\end{equation}
where $x_0\in\mathbb R$, $W = (W_t)_{t\in [0,2T]}$ is a Brownian motion, $b,\sigma\in C^1(\mathbb R)$, and $b'$ and $\sigma'$ are bounded. Under these conditions on $b$ and $\sigma$, Equation (\ref{main_equation}) has a unique (strong) solution $X = (X_t)_{t\in [0,2T]}$. Under additional conditions, the transition density $p_t(x,.)$ is well defined and can be understood as the conditional density of $X_{s + t}$ given $X_s = x$. The two time indexes $s+t$ and $s$ explain why the process is considered on $[0,2T]$, while each index evolves in $[0,T]$. Our aim is to provide and study a nonparametric estimator of the function $(x,y)\mapsto p_t(x,y)$ ($t\in [0,T]$), computed from $N$ independent copies of $X$ observed on the time-interval $[0,2T]$.\\
Let us start with some comments about the observation scheme. It is related to functional data analysis, dealing with samples of infinite dimensional data (see Ramsay and Silverman \cite{RS07} and Wang {\it et al.} \cite{WCM16}). Several examples of applications in biology and meteorology (resp. energy) are developed in Ramsay and Silverman \cite{RS07} (resp. Candanedo {\it et al.} \cite{CANDANEDO17}). In Section \ref{section_motivating_examples}, we also present an example of financial model leading to this observation setting.\\
Since few years, statistical inference from copies of diffusion processes, especially estimators of the drift function $b$, has been deeply investigated. In that way, tools related to statistical inference for stochastic processes have been used in the world of functional data analysis. From independent copies of the solution $X$ of an SDE, projection least squares estimators of $b$ have been studied in Comte and Genon-Catalot \cite{CGC20} for continuous time observations, in Denis {\it et al.} \cite{DDBM21} for discrete time observations with a classification purpose in a parametric setting, and in Denis {\it et al.} \cite{DDBM20} for the same question in the nonparametric framework. Marie and Rosier \cite{MR22} propose a kernel based Nadaraya-Watson estimator of the drift function $b$, with bandwidth selection relying on the Penalized Comparison to Overfitting criterion recently introduced in Lacour {\it et al.} \cite{LMR17}. Still for independent copies of $X$, Halconruy and Marie \cite{HM23} investigate the properties of the projection least squares estimator of $b$ when the Brownian motion $W$ is replaced by a L\'evy process in Equation (\ref{main_equation}), and Comte and Genon-Catalot \cite{CGC23b} and Marie \cite{MARIE23} deal with estimators of the drift function in non-autonomous SDE. More recently, copies-based estimation with dependency has been investigated. For instance, one may refer to Della Maestra and Hoffmann \cite{DMH22} or Belomestny {\it et al.} \cite{BPP21}, dealing with nonparametric estimators in interacting particle systems and McKean-Vlasov models. However, the question of transition density estimation has not been considered in this setting, while this function is a key quantity of the process.
\\
\\
A transition density is a conditional density, but it is beyond our scope to describe all works dealing with the nonparametric estimation of the conditional density of $Y$ given $X=x$ when $N$ i.i.d. observations of the couples $(X_i,Y_i)$ are available. We refer to the recent paper by Comte and Lacour \cite{CL23}, and the references therein, on this topic. The question of estimating a transition density has already been studied, by several authors, for discrete samples of general Markov chains. The discrete time observation of one path of the solution $X$ of Equation (\ref{main_equation}), say $X_{k\Delta}$ with $k = 1,\dots,n$ and $\Delta > 0$, fits to this framework. Nonparametric estimation strategies can be based on contrast minimization, as in Lacour \cite{LACOUR07}, an approach we also follow and adapt to our setting. A transition can also be seen as a ratio of two densities, namely the density of $(X_s,X_{s + t})$ divided by that of $X_s$, and then the ratio of the corresponding density estimators may be proposed and studied (see Lacour \cite{LACOUR08}). The paper by Sart \cite{SART14} presents and studies estimators corresponding to the two strategies. Recently, in higher dimension, L\"offler and Picard \cite{LP21} estimate a general transition operator relying on singular value decomposition.
\\
\\
The framework considered in the present work is new. Indeed, the question of estimating the transition density $p_t(x,y)$ ($t\in (0,T]$) from the observation on $[0,2T]$ of $N$ independent copies $X^i :=\mathcal I(x_0,W^i)$ ($i\in\{1,\dots,N\}$) of the solution $X$ of Equation (\ref{main_equation}), where $\mathcal I(\cdot)$ is the It\^o map for Equation (\ref{main_equation}) and $W^1,\dots,W^N$ are $N$ independent copies of the Brownian motion $W$, has not been yet investigated. Let us consider $\mathcal S_{\bf m} :=\mathcal S_{\varphi,m_1}\otimes\mathcal S_{\psi,m_2}$, where $\mathcal S_{\varphi,m_1} := {\rm span}\{\varphi_1,\dots,\varphi_{m_1}\}$ (resp. $\mathcal S_{\psi,m_2} := {\rm span}\{\psi_1,\dots,\psi_{m_2}\}$), $\varphi_1,\dots,\varphi_{N_T}$ (resp. $\psi_1,\dots,\psi_{N_T}$) are continuous functions from $I$ (resp. $J$) into $\mathbb R$ such that $(\varphi_1,\dots,\varphi_{N_T})$ (resp. $(\psi_1,\dots,\psi_{N_T})$) is an orthonormal family in $\mathbb L^2(I,dx)$ (resp. $\mathbb L^2(J,dx)$), $N_T := [NT] + 1$ and $I,J\subset\mathbb R$ are non-empty intervals. In this paper, we study the following projection least squares estimator of the transition density function $p_t$:
\begin{displaymath}
\widehat p_{\mathbf m,t}(x,y) :=
\sum_{j = 1}^{m_1}\sum_{k = 1}^{m_2}
\widehat\Theta_{j,k}\varphi_j(x)\psi_k(y),
\end{displaymath}
where $\mathbf m = (m_1,m_2)\in\{1,\dots,N_T\}^2$, and the coefficients $\widehat\Theta_{j,k}$ are computed from the observed paths of the processes $(X_{t}^{i})_{0\leqslant t\leqslant 2T}$ ($i\in\{1,\dots,N\}$). This estimator is the minimizer over functions $\tau\in\mathcal S_{\bf m}$ of an objective function defined as the empirical version of
\begin{displaymath}
\|\tau - p_t\|_{f}^{2} =
\int_{-\infty}^{\infty}\int_{-\infty}^{\infty}
(\tau(x,y) - p_t(x,y))^2f(x)dxdy;
\end{displaymath}
the $f$-weighted distance between $\tau$ and $p_t$ with
\begin{displaymath}
f(\cdot) =\frac{1}{T}
\int_{0}^{T}p_s(x_0,\cdot)ds.
\end{displaymath}
Roughly speaking, $\widehat p_{\mathbf m,t}$ mimics the projection of $p_t$ on $\mathcal S_{\bf m}$ with respect to the $f$-weighted norm $\|.\|_f$, and then our purpose is to prove that the empirical norm and the empirical criterion defining $\widehat p_{\mathbf m,t}$ are near of their theoretical counterparts with high probability. In this way, we are able to evaluate the quadratic risk of the estimator.\\
A risk bound is established on $\widehat p_{\mathbf m,t}$, for a fixed $\mathbf m$. Under regularity assumptions and an adequate choice of $\mathbf m = (m_1,m_2)$, a rate of convergence is provided for the projection least squares estimator of $p_t$. Then, we focus on the example where $p_t$ belongs to an anisotropic $2$-dimensional Sobolev-Hermite space. Clearly, other regularity spaces may be considered, but since the Hermite basis is $\mathbb R$-supported, this example is very instructive and allows us to rely on some results already established in Comte and Lacour \cite{CL23} for the estimation of the conditional density. The established rates are achieved when $m_1$ and $m_2$ - the coordinates of $\mathbf m$ - are chosen as $N$ to a power depending on the regularity orders of $p_t$. Since these regularity parameters are unknown, the previous choice of $\mathbf m$ cannot be done in practice. This is why a data-driven selection procedure is provided and studied, allowing us to select a model $\widehat{\bf m} = (\widehat m_1,\widehat m_2)$ directly from the data. The couple of dimensions $(\widehat m_1,\widehat m_2)$ is selected in a random subset of $\{1,\dots,N_T\}^2$ thanks to a penalized version of the estimation contrast, which corresponds to the idea of a data-driven squared-bias/variance compromise. Then, a risk bound is established on the adaptive estimator $\widehat p_{\widehat{\bf m},t}$, which automatically achieves the squared-bias/variance tradeoff up to a $\log(N)$ factor. We also prove that the bound can be improved, and the $\log(N)$-loss corrected, in two special cases: when $t > t_0$ for a given $t_0 > 0$, and when $t > 0$ with $I$ - the support of the basis $\varphi$ - compact.
\\
\\
To start with, in a short preliminary Section \ref{section_motivating_examples}, we propose two application settings where our observation scheme is meaningful, and where our estimator of the transition density may be used. The projection least squares estimator is defined in Section \ref{section_projection_LS_estimator}, and non-adaptive risk bounds on a given model are established in Section \ref{section_risk_bounds}. A model selection procedure and an adaptive risk bound are provided in Section \ref{section_model_selection}. The whole procedure is illustrated through simulations in Section \ref{section_numerical_experiments}. Finally, Section \ref{section_conclusion} provides concluding remarks and proofs are gathered in Appendix \ref{section_proofs}.
\\
\\
\textbf{Notation and basic definitions:}
\begin{itemize}
 \item The space of $m_1\times m_2$ matrices with coefficients in $\mathbb R$ is denoted by $\mathcal M_{m_1,m_2}(\mathbb R)$.
 \item For every $M\in\mathcal M_{m_1,m_2}(\mathbb R)$, the transpose of $M$ is denoted by $M^*$, and the largest eigenvalue of $MM^*$ is denoted by $\|M\|_{\rm op}^{2}$.
 \item The usual inner product (resp. norm) on $\mathbb L^2(\mathbb R^2)$ is denoted by $\langle .,.\rangle$ (resp. $\|.\|$). Recall that, for every $\tau,\tau^{\star}\in\mathbb L^2(\mathbb R^2)$,
 \begin{displaymath}
 \langle\tau,\tau^{\star}\rangle =
 \int_{\mathbb R^2}\tau(x,y)\tau^{\star}(x,y)dxdy.
 \end{displaymath}
 For the sake of readability, the usual inner product and the associated norm on $\mathbb L^2(\mathbb R)$ are denoted in the same way.
\end{itemize}
%

% Section : Two motivating examples.

%
\section{Two motivating examples}\label{section_motivating_examples}
In this section, we briefly present two possible applications of our estimation method of $p_t$.
\begin{itemize}
 \item Assume that $\sigma(\cdot)^2 > 0$, and consider the parabolic partial differential equation defined by
 \begin{equation}\label{PDE}
 \frac{\partial u}{\partial t}(t,x) +
 \frac{1}{2}\sigma(x)^2\frac{\partial^2u}{\partial x^2}(t,x) +
 b(x)\frac{\partial u}{\partial x}(t,x) = 0,\quad
 u(T,x) = v(x),
 \end{equation}
 where $v :\mathbb R\rightarrow\mathbb R$ is a known twice continuously differentiable function. Moreover, let $X^{t,x}$ be the solution of Equation (\ref{main_equation}) starting from $x\in\mathbb R$ at time $t\in [0,T)$. By Lamberton and Lapeyre \cite{LL08}, Theorem 5.1.7, the solution of Equation (\ref{PDE}) is given by
 \begin{displaymath}
 F(t,x) :=\mathbb E(\varphi(X_{T}^{t,x})) =
 \int_{-\infty}^{\infty}v(y)p_{T - t}(x,y)dy.
 \end{displaymath}
 Thus, the solution of Equation (\ref{PDE}) can be estimated by
 \begin{equation}\label{Monte_Carlo_PDE}
 \widehat F_{\mathbf m}(t,x) :=
 \int_{-\infty}^{\infty}v(y)\widehat p_{\mathbf m,T - t}(x,y)dy.
 \end{equation}
 In the same spirit as Milstein {\it et al.} \cite{MSS04}, (\ref{Monte_Carlo_PDE}) provides a Monte-Carlo method which may replace the usual finite difference algorithm to solve Equation (\ref{PDE}) numerically. Note that in this example, $b$, $\sigma$ and $v$ are assumed to be known, and then one may simulate a large amount $N$ of copies of $X$ observed at high-frequency on $[0,2T]$ in order to compute $\widehat p_{\mathbf m,T - t}$, and then $\widehat F_{\bf m}(t,x)$. In other words, the copies-based observation scheme is appropriate in this context.
 \item The options pricing in finance is another possible application of our estimation method of $p_t$. Let $X = (X_t)_{t\in\mathbb R}$ be the prices process of a risky asset, which risk-neutral dynamics are modeled by
 \begin{equation}\label{prices_model}
 \frac{dX_t}{X_t} = (r -\delta)dt +\sigma(X_t)dW_t,
 \end{equation}
 where $r > 0$ is the risk-free rate, $\delta > 0$ the dividend rate, $W = (W_t)_{t\in\mathbb R}$ is a two-sided Brownian motion, and $X_0$ is a $\sigma((W_t)_{t\in\mathbb R_-})$-measurable square integrable random variable. Consider also the option with payoff $v(X_{T}^{0,x})$, which price $P(x,T)$ satisfies
 \begin{displaymath}
 P(x,T) =
 e^{-rT}\mathbb E(v(X_{T}^{0,x})) =
 e^{-rT}\int_{0}^{\infty}v(y)p_T(x,y)dy.
 \end{displaymath}
 Thus, when $r <\delta$, an estimator of $P(x,T)$ is given by
 \begin{displaymath}
 \widehat P_{\bf m}(x,T) :=
 e^{-rT}\int_{0}^{\infty}v(y)\widehat p_{\mathbf m,T}(x,y)dy,
 \end{displaymath}
 where the copies $X^1,\dots,X^N$ of $(X_{t}^{0,x})_{t\in [0,2T]}$ are here constructed from $(X_t)_{t\in\mathbb R_-}$, the past of the (recurrent Markov) prices process, by following the same line as in Marie \cite{MARIE24}, Remark 2.3. So, the copies-based observation scheme remains appropriate in this context, but requires high-frequency data on a long time-interval.
\end{itemize}
Note that risk bounds on truncated versions of $\widehat F_{\bf m}(t,x)$ and $\widehat P_{\bf m}(x,T)$ may be deduced from the results of Section \ref{section_f_weighted_risk_bound_truncated}, but are omitted since they may not be optimal.
%

% Section : A projection least squares estimator of the transition density function.

%
\section{A projection least squares estimator of the transition density function}\label{section_projection_LS_estimator}
%

% Subsection : Assumptions on the coefficients of the main equation.

% 
\subsection{Assumptions on the coefficients of Equation (\ref{main_equation})}\label{section_assumptions}
Throughout the paper, we consider copies $X^1,\dots,X^N$ of the solution $X$ of Equation (\ref{main_equation}) under the condition:
\begin{equation}\label{basic_condition}
x_0\in\mathbb R\textrm{, }
W = (W_t)_{t\in [0,2T]}\textrm{ is a Brownian motion, }
b,\sigma\in C^1(\mathbb R)\textrm{ and }
b',\sigma'\textrm{ are bounded.}
\end{equation}
We also assume that $\sigma$ satisfies the following non-degeneracy condition:
\begin{equation}\label{nondegeneracy_condition}
\exists\alpha,A > 0 :
\forall x\in\mathbb R\textrm{, }
\alpha\leqslant |\sigma(x)|\leqslant\alpha + A.
\end{equation}
As already mentioned, under condition (\ref{basic_condition}), Equation (\ref{main_equation}) has a unique solution. Under the condition (\ref{nondegeneracy_condition}), the transition density function $p_t$ is well-defined. Conditions (\ref{basic_condition}) and (\ref{nondegeneracy_condition}) are assumed to be fulfilled throughout the paper. Some properties of the density $p_t(x_0,.)$ of $X_t$ are provided in the following proposition.
%

% Proposition : Properties of the transition density.

%
\begin{proposition}\label{properties_transition}
There exist two positive constants $\mathfrak c_T$ and $\mathfrak m_T$, depending on $T$ but not on $t$, such that for every $x,y\in\mathbb R$,
\begin{equation}\label{Gaussian_bound_density_IC}
p_t(x_0,x)\leqslant
\mathfrak c_Tt^{-\frac{1}{2}}\exp\left(-\mathfrak m_T\frac{(x - x_0)^2}{t}\right).
\end{equation}
Then, the density function
\begin{equation}\label{properties_transition_1}
f(\cdot) :=\frac{1}{T}\int_{0}^{T}p_s(x_0,\cdot)ds
\end{equation}
is well-defined and satisfies the following properties:
\begin{enumerate}
 \item $f$ is upper bounded on $\mathbb R$.
 \item For every compact interval $I\subset\mathbb R$, there exists a positive constant $\underline{\mathfrak m} =\underline{\mathfrak m}(I)$ such that $f(\cdot)\geqslant\underline{\mathfrak m}$ on $I$.
 \item For every $\kappa\in\mathbb R_+$, $|b|^{\kappa}\in\mathbb L^2(\mathbb R,f(x)dx)$.
\end{enumerate}
\end{proposition}
%

% Proof.

%
\begin{proof}
First, under the conditions (\ref{basic_condition}) and (\ref{nondegeneracy_condition}), by Menozzi {\it et al.} \cite{MPZ21}, Theorem 2.1, there exist two positive constants $\overline{\mathfrak c}_T$ and $\overline{\mathfrak m}_T$, depending on $T$ but not on $t$, such that for every $x,y\in\mathbb R$,
\begin{equation}\label{Gaussian_bound_density}
p_t(x,y)\leqslant
\overline{\mathfrak c}_Tt^{-\frac{1}{2}}
\exp\left(-\overline{\mathfrak m}_T\frac{(y -\theta_t(x))^2}{t}\right),
\end{equation}
where $s\mapsto\theta_s(x)$ is the solution of the differential equation
\begin{displaymath}
z_s = x +\int_{0}^{s}b(z_u)du
\textrm{ $;$ }s\in [0,T].
\end{displaymath}
For any $u\in (0,T]$ and $x\in\mathbb R$,
\begin{displaymath}
|x - x_0|\leqslant
|x -\theta_u(x_0)| + uS_T(x_0)
\quad {\rm with}\quad
S_T(x_0) =\sup_{s\in [0,T]}|b(\theta_s(x_0))|,
\end{displaymath}
leading to
\begin{displaymath}
-\frac{1}{2}(x - x_0)^2 + u^2S_T(x_0)^2\geqslant
-(x -\theta_u(x_0))^2.
\end{displaymath}
Thus, by Inequality (\ref{Gaussian_bound_density}),
\begin{eqnarray*}
 p_t(x_0,x) & \leqslant &
 \mathfrak c_Tt^{-\frac{1}{2}}\exp\left(-\mathfrak m_T\frac{(x - x_0)^2}{t}\right)\\
 \nonumber
 & &
 \hspace{2cm}{\rm with}\quad
 \mathfrak c_T =\overline{\mathfrak c}_Te^{\overline{\mathfrak m}_TTS_{T}(x_0)^2}
 \quad {\rm and}\quad
 \mathfrak m_T =\frac{\overline{\mathfrak m}_T}{2}.
\end{eqnarray*}
In particular, $t\mapsto p_t(x_0,x)$ belongs to $\mathbb L^1([0,T])$, which legitimates to consider the density function $f$ defined by (\ref{properties_transition_1}). Still by Inequality (\ref{Gaussian_bound_density_IC}):
\begin{itemize}
 \item $f$ is upper bounded on $\mathbb R$. Indeed, for every $x\in\mathbb R$,
 \begin{displaymath}
 f(x)\leqslant
 \frac{1}{T}\int_{0}^{T}\mathfrak c_Ts^{-\frac{1}{2}}ds =
 2\mathfrak c_TT^{-\frac{1}{2}}.
 \end{displaymath}
 \item Since $b'$ is bounded (and then $b$ has linear growth),
 \begin{displaymath}
 |b|^{\kappa}\in\mathbb L^2(\mathbb R,f(x)dx)
 \textrm{ $;$ }
 \forall\kappa\in\mathbb R_+.
 \end{displaymath}
 Indeed, for every $\upsilon\in\mathbb R_+$ and $t\in (0,T]$,
 \begin{eqnarray*}
  \mathbb E(|X_t|^{\upsilon})
  & = &
  \int_{-\infty}^{\infty}|x|^{\upsilon}p_t(x_0,x)dx\\
  & \leqslant &
  t^{-\frac{1}{2}}\underbrace{
  \mathfrak c_T\int_{-\infty}^{\infty}|x|^{\upsilon}
  \exp\left(-\mathfrak m_T\frac{(x - x_0)^2}{T}\right)dx}_{=:\mathfrak c_{T,\upsilon} <\infty}
 \end{eqnarray*}
 and then, for every $\kappa\in\mathbb R_+$,
 \begin{eqnarray*}
  \int_{-\infty}^{\infty}|b(x)|^{2\kappa}f(x)dx & = &
  \frac{1}{T}\int_{0}^{T}\mathbb E(|b(X_s)|^{2\kappa})ds\\
  & \leqslant &
  \mathfrak c_1\left(1 +\frac{1}{T}\int_{0}^{T}\mathbb E(|X_s|^{2\kappa})ds\right)
  \leqslant\mathfrak c_1(1 + 2\mathfrak c_{T,2\kappa}T^{-\frac{1}{2}}) <\infty,
 \end{eqnarray*}
 where $\mathfrak c_1$ is a positive constant depending only on $b$ and $\kappa$.
\end{itemize}
Now, Menozzi {\it et al.} \cite{MPZ21}, Theorem 1.2 also says that there exist two positive constants $\underline{\mathfrak c}_T$ and $\underline{\mathfrak m}_T$, depending on $T$ but not on $t$, such that for every $x,y\in\mathbb R$,
\begin{equation}\label{Gaussian_lower_bound_density}
p_t(x,y)\geqslant
\underline{\mathfrak c}_Tt^{-\frac{1}{2}}
\exp\left(-\underline{\mathfrak m}_T\frac{(y -\theta_t(x))^2}{t}\right).
\end{equation}
For any $u\in (0,T]$ and $x\in\mathbb R$,
\begin{displaymath}
|x -\theta_u(x_0)|\leqslant
|x - x_0| + uS_T(x_0),
\end{displaymath}
leading to
\begin{displaymath}
-(x -\theta_u(x_0))^2\geqslant
-2(x - x_0)^2 - 2u^2S_T(x_0)^2.
\end{displaymath}
So, by Inequality (\ref{Gaussian_lower_bound_density}),
\begin{eqnarray*}
 p_t(x_0,x)
 & \geqslant &
 \mathfrak c_{T}^{*}t^{-\frac{1}{2}}
 \exp\left(-\mathfrak m_{T}^{*}\frac{(x - x_0)^2}{t}\right)\\
 & &
 \hspace{2cm}{\rm with}\quad
 \mathfrak c_{T}^{*} =\underline{\mathfrak c}_Te^{-2\underline{\mathfrak m}_TTS_{T}(x_0)^2}
 \quad {\rm and}\quad
 \mathfrak m_{T}^{*} = 2\underline{\mathfrak m}_T
\end{eqnarray*}
and then, for every compact interval $I\subset\mathbb R$ and $x\in I$,
\begin{eqnarray*}
 f(x) & \geqslant &
 \frac{\mathfrak c_{T}^{*}}{T}\int_{T/2}^{T}s^{-\frac{1}{2}}\exp\left(-\frac{
 \mathfrak m_{T}^{*}}{s}(x - x_0)^2\right)ds
 \nonumber\\
 & \geqslant &
 \underline{\mathfrak m}(I) :=
 \frac{\mathfrak c_{T}^{*}}{T}
 \exp\left(-\frac{2\mathfrak m_{T}^{*}}{T}
 \max\{(\max(I) - x_0)^2,(\min(I) - x_0)^2\}\right)
 \int_{T/2}^{T}s^{-\frac{1}{2}}ds > 0.
\end{eqnarray*}
This concludes the proof.
\end{proof}
%

% Subsection : The projection least squares estimator and some related definitions.

%
\subsection{The projection least squares estimator and some related definitions}
The objective function $\gamma_N$ is defined by
\begin{displaymath}
\gamma_N(\tau) :=
\frac{1}{NT}\sum_{i = 1}^{N}\left(\int_{0}^{T}\int_{-\infty}^{\infty}\tau(X_{s}^{i},y)^2dyds -
2\int_{0}^{T}\tau(X_{s}^{i},X_{s + t}^{i})ds\right)
\textrm{ $;$ }\tau\in\mathcal S_{\bf m}.
\end{displaymath}
In order to explain the relevance of the criterion, let us compute its expectation:
\begin{displaymath}
\mathbb E(\gamma_N(\tau)) =
\int_{-\infty}^{\infty}\int_{-\infty}^{\infty}
(\tau(x,y) - p_t(x,y))^2f(x)dxdy -
\int_{-\infty}^{\infty}\int_{-\infty}^{\infty}
p_t(x,y)^2f(x)dxdy.
\end{displaymath}
This shows that, the closer $\tau$ is to $p_t$, the smaller $\mathbb E(\gamma_N(\tau))$, and that $\gamma_N(\tau)$ is the empirical version of the $f$-weighted $\mathbb L^2$-distance between $\tau$ and $p_t$. This is the reason why we define the estimator of $p_t$ as the minimizer of $\gamma_N$ in $\mathcal S_{\bf m}$.
\\
\\
{\bf Remark.} Note that any $\tau\in\mathcal S_{\bf m}$ can be written $\tau =\sum_{j,\ell}\Theta_{j,\ell}(\varphi_j\otimes\psi_{\ell})$ with $\Theta\in\mathcal M_{m_1,m_2}(\mathbb R)$. Then, $\tau$ is characterized by the matrix $\Theta$.
\\
\\
First, let us study the conditions ensuring that $\gamma_N$ has a unique minimizer in $\mathcal S_{\bf m}$. For $\tau =\sum_{j,\ell}\Theta_{j,\ell}(\varphi_j\otimes\psi_{\ell})$ with $\Theta\in\mathcal M_{m_1,m_2}(\mathbb R)$,
\begin{displaymath}
\nabla_{\tau}\gamma_N(\tau) =
2(\widehat\Psi_{m_1}\Theta -\widehat Z_{\mathbf m,t}),
\end{displaymath}
where
\begin{equation}\label{empirical_Psi_matrix}
\widehat\Psi_{m_1} :=\left(
\frac{1}{NT}\sum_{i = 1}^{N}\int_{0}^{T}\varphi_j(X_{s}^{i})\varphi_{j'}(X_{s}^{i})ds
\right)_{j,j'\in\{1,\dots,m_1\}}
\end{equation}
and
\begin{equation}\label{Z_vector}
\widehat Z_{\mathbf m,t} :=
\left(\frac{1}{NT}\sum_{i = 1}^{N}\int_{0}^{T}\varphi_j(X_{s}^{i})\psi_{\ell}(X_{s + t}^{i})ds
\right)_{(j,\ell)\in\{1,\dots,m_1\}\times\{1,\dots,m_2\}}.
\end{equation}
The symmetric matrix $\widehat\Psi_{m_1}$ is positive semidefinite because for any $y\in\mathbb R^{m_1}$,
\begin{displaymath}
y^*\widehat\Psi_{m_1}y =
\frac{1}{NT}\sum_{i = 1}^{N}\int_{0}^{T}
\left(\sum_{j = 1}^{m_1}y_j\varphi_j(X_{s}^{i})\right)^2ds\geqslant 0.
\end{displaymath}
If in addition $\widehat{\Psi}_{m_1}$ is invertible, it is positive definite, and then
\begin{equation}\label{definition_projection_LS}
\widehat p_{\mathbf m,t} =\sum_{j = 1}^{m_1}\sum_{\ell = 1}^{m_2}
[\widehat\Theta_{\mathbf m,t}]_{j,\ell}(\varphi_j\otimes\psi_{\ell})
\quad {\rm with}\quad
\widehat\Theta_{\mathbf m,t} =\widehat\Psi_{m_1}^{-1}\widehat Z_{\mathbf m,t}
\end{equation}
is the only minimizer of $\gamma_N$ in $\mathcal S_{\bf m}$ called the projection least squares estimator of $p_t$.
\\
\\
Now, let us consider the theoretical counterpart $\Psi_{m_1} :=\mathbb E(\widehat\Psi_{m_1})$ of $\widehat\Psi_{m_1}$ (see Equality (\ref{empirical_Psi_matrix})), the empirical inner product $\langle .,.\rangle_N$ such that
\begin{displaymath}
\langle\tau,\tau^{\star} \rangle_N :=
\frac{1}{NT}\sum_{i = 1}^{N}
\int_{0}^{T}\int_{-\infty}^{\infty}
\tau(X_{s}^{i},y)\tau^{\star}(X_{s}^{i},y)dyds
\textrm{ $;$ }
\forall\tau,\tau^{\star}\in\mathcal S_{\bf m},
\end{displaymath}
and the $f$-weighted inner product $\langle .,.\rangle_f$ such that
\begin{displaymath}
\langle\tau,\tau^{\star}\rangle_f :=
\mathbb E(\langle\tau,\tau^{\star}\rangle_N) =
\int_{-\infty}^{\infty}\int_{-\infty}^{\infty}
\tau(x,y)\tau^{\star}(x,y)f(x)dxdy
\textrm{ $;$ }
\forall\tau,\tau^{\star}\in\mathcal S_{\bf m}.
\end{displaymath}
The empirical norm associated to $\langle .,.\rangle_N$ (resp. $\langle .,.\rangle_f$) is denoted by $\|.\|_N$ (resp. $\|.\|_f$). For any $\tau\in\mathcal S_{\bf m}$,
\begin{displaymath}
\gamma_N(\tau) =\|\tau\|_{N}^{2} -
\frac{2}{NT}\sum_{i = 1}^{N}\int_{0}^{T}
\tau(X_{s}^{i},X_{s + t}^{i})ds
\end{displaymath}
and then, interestingly,
\begin{displaymath}
\mathbb E(\gamma_N(\tau)) =
\|\tau\|_{f}^{2} - 2\langle\tau,p_t\rangle_f,
\end{displaymath}
justifying again the choice of the objective function $\gamma_N$ to define our estimator $\widehat p_{\mathbf m,t}$ of $p_t$. It is standard in nonparametric regression-type settings to have to handle a random norm depending on the observations, and to control the proximity with its expectation thanks to deviation inequalities.
%

% Subsection : Examples of bases.

%
\subsection{Examples of bases}\label{section_bases}
In addition to the empirical norm, related to the data, and its theoretical counterpart - the $f$-weighted $\mathbb L^2$-norm $\|.\|_f$ -, the usual $\mathbb L^2$-norm $\|.\|$ defined at the end of the introduction section appears through the bases and their characteristics.
\\
\\
First, let $u = (u_j)_{j\geqslant 1}$ be an orthonormal $I$-supported family of $\mathbb L^2(\mathbb R,dx)$ (equipped with $\langle .,.\rangle$), and consider $\mathcal S_{u,m} := {\rm span}\{u_1,\dots,u_m\}$; the linear space generated by the $m\in\mathbb N^*$ first functions of the family $u$. Consider also
\begin{displaymath}
\mathfrak L_u(m) :=
\sup_{x\in I}\sum_{j = 1}^{m}u_j(x)^2.
\end{displaymath}
One may easily check that
\begin{displaymath}
\mathfrak L_u(m) =
\sup_{\tau\in\mathcal S_{u,m} :\|\tau\| = 1}\|\tau\|_{\infty}
\quad\textrm{(see Lemma 1 in Big\'e and Massart \cite{BM98}, p. 337).}
\end{displaymath}
So, $\mathfrak L_u(m)$ measures the link between the $\mathbb L^2$ and $\mathbb L^{\infty}$ norms in the function space $\mathcal S_{u,m}$: when $I = [0,1]$,
\begin{displaymath}
\|\tau\|^2\leqslant
\|\tau\|_{\infty}^{2}\leqslant
\mathfrak L(m)\|\tau\|^2
\textrm{ $;$ }
\forall\tau\in\mathcal S_{u,m}.
\end{displaymath}
Now, let us provide examples of usual bases $u$, and controls of $\mathfrak L_u(m)$ for each of them:
\begin{itemize}
 \item The trigonometric basis is defined on $I = [0,1]$ by
 \begin{displaymath}
 {\rm trig}_1(x) := 1,\quad
 {\rm trig}_{2j}(x) :=\sqrt 2\cos(2\pi jx),\quad
 {\rm trig}_{2j + 1}(x) :=\sin(2\pi jx)
 \end{displaymath}
 for $j = 1,\dots,(m - 1)/2$ with $m\in\mathbb N\backslash (2\mathbb N)$. Since $\cos(\cdot)^2 +\sin(\cdot)^2 = 1$, $\mathfrak L_{\rm trig}(m) = m$. In the same way, $\mathfrak L_u(m)$ is of order $m$ for splines, wavelets, or any orthonormal bounded bases.
 \item For the Legendre basis $(\ell_j)_{j\geqslant 1}$, defined via normalized Legendre's polynomials on $I = [-1,1]$ (see Tsybakov \cite{TSYBAKOV09}, Section 1.2.2), $\mathfrak L_{\ell}(m)\lesssim m^{2}$ as established in Cohen {\it et al.} \cite{CDL13} (p. 831).
 \item Finally, let us focus on the Hermite basis $(h_j)_{j\in\mathbb N}$, defined on $I =\mathbb R$ by
 \begin{equation}\label{Hermite_basis}
 h_j(x) :=\mathfrak c_jH_j(x)e^{-\frac{x^2}{2}}
 \quad {\rm with}\quad
 \mathfrak c_j = (2^jj!\sqrt\pi)^{-\frac{1}{2}}
 \quad {\rm and}\quad
 H_j(x) = (-1)^je^{x^2}\frac{d^j}{dx^j}e^{-x^2}
 \end{equation}
 for every $x\in\mathbb R$ and $j\in\mathbb N$. The sequence $(h_j)_{j\in\mathbb N}$ is an orthonormal family of $\mathbb L^2(\mathbb R,dx)$. Moreover, the $h_j$'s are bounded by $\pi^{-1/4}$, and by Lemma 1 in Comte and Lacour \cite{CL23}, there exists a constant $\mathfrak c_h > 0$ such that
 \begin{equation}\label{L_Hermite_basis}
 \mathfrak L_h(m)\leqslant\mathfrak c_h\sqrt m.
 \end{equation}
\end{itemize}
Without loss of generality, in the sequel, we assume that $\mathfrak L_u(m)\geqslant 1$. Otherwise, $\mathfrak L_u(m)$ can be replaced by $1\vee\mathfrak L_u(m)$.
%

% Section : Risk bounds on the projection least squares estimator.

%
\section{Risk bounds on the projection least squares estimator}\label{section_risk_bounds}
%

% Subsection : Risk bound with respect to the empirical norm.

%
\subsection{Risk bound with respect to the empirical norm}\label{section_empirical_risk_bound}
This subsection deals with a nonadaptive risk bound, with respect to the empirical norm $\|.\|_N$, on our estimator $\widehat p_{\mathbf m,t}$ of $p_t$, built in $\mathcal S_{\bf m} =\mathcal S_{\varphi,m_1}\otimes\mathcal S_{\psi,m_2}$. In the sequel, $\mathbf m = (m_1,m_2)$ fulfills the following assumption.
%

% Assumption : Condition on \mathbf m.

%
\begin{assumption}\label{condition_m}
There exist two constants $\mathfrak c_{\ref{condition_m}} > 0$ and $q\in\mathbb N^*$, not depending on $\mathbf m$ and $N$, such that
\begin{equation}\label{weak_limit_psi}
\mathfrak L_{\psi}(m_2)\leqslant\mathfrak c_{\ref{condition_m}}N^q.
\end{equation}
Moreover,
\begin{equation}\label{stability_condition}
\mathfrak L_{\varphi}(m_1)(\|\Psi_{m_1}^{-1}\|_{\rm op}\vee 1)
\leqslant\frac{\mathfrak c_{\Lambda}}{2}\cdot\frac{NT}{\log(NT)}
\end{equation}
with
\begin{displaymath}
\mathfrak c_{\Lambda} =\frac{1 -\log(2)}{(1 + p)T}
\quad\textrm{and}\quad
p = 2(q + 4) + 1.
\end{displaymath}
\end{assumption}
\noindent
The first part (\ref{weak_limit_psi}) of Assumption \ref{condition_m} is a weak limit on the maximal dimension $m_2$ that can be considered (see Section \ref{section_bases}). The second part (\ref{stability_condition}) of Assumption \ref{condition_m} is a generalization of the so-called {\it stability condition} introduced in the nonparametric regression framework in Cohen {\it et al.} \cite{CDL13}, and already extended to the context of independent copies of continuous diffusion processes framework in Comte and Genon-Catalot \cite{CGC20}. In fact, (\ref{stability_condition}) is in turn a limit on the maximal dimension $m_1$ that can be considered. It can be stronger than $m_1\leqslant NT$, but to what extent? First, let us answer this question when $I$ is a compact interval. By Proposition \ref{properties_transition}.(2), there exists $\underline{\mathfrak m} > 0$ such that $f(\cdot)\geqslant\underline{\mathfrak m}$ on $I$, and then
\begin{eqnarray*}
 \|\Psi_{m_1}^{-1}\|_{\rm op} & = &
 \frac{1}{\lambda_{\min}(\Psi_{m_1})} =
 \left(\inf_{\theta :\|\theta\|_{2,m_1} = 1}\sum_{j,j' = 1}^{m_1}\theta_j\theta_{j'}[\Psi_{m_1}]_{j,j'}\right)^{-1}\\
 & = &
 \left[\inf_{\theta :\|\theta\|_{2,m_1} = 1}\int_I\left(
 \sum_{j = 1}^{m_1}\theta_j\varphi_j(x)\right)^2f(x)dx\right]^{-1}\leqslant\frac{1}{\underline{\mathfrak m}}
 \quad\textrm{ for every $m_1\in\{1,\dots,N_T\}$.}
\end{eqnarray*}
Moreover, assume that there exists a constant $\mathfrak c_{\varphi} > 0$, not depending on $m_1$ and $N$, such that $\mathfrak L_{\varphi}(m_1)\leqslant\mathfrak c_{\varphi}^{2}m_1$. For instance, all bounded bases satisfy this condition. Then, (\ref{stability_condition}) is fulfilled by any
\begin{displaymath}
m_1\leqslant\frac{\mathfrak c_{\Lambda}}
{2\mathfrak c_{\varphi}^{2}(\underline{\mathfrak m}^{-1}\vee 1)}
\cdot\frac{NT}{\log(NT)}.
\end{displaymath}
In other words, when $I$ is compact and the $\varphi_j$'s are bounded (i.e. $\max_j\|\varphi_j\|_{\infty}\leqslant\Phi$ with $\Phi > 0$), $m_1$ needs to be of order $NT/\log(NT)$, which remains a weak condition. Now, when $I$ is a non-compact interval, the set of possible values for $m_1$ becomes smaller. For instance, for Hermite's basis, it holds that $\|\Psi_{m_1}^{-1}\|_{\rm op}\gtrsim\sqrt{m_1}$ by Comte and Genon-Catalot \cite{CGC19}, Proposition 8, which necessarily reduces the collection of authorized values for $m_1$.
\\
\\
Under Assumption \ref{condition_m}, we can prove the following risk bound.
%

% Theorem : Empirical risk bound.

%
\begin{theorem}\label{empirical_risk_bound}
Consider $p_{I\times J,t} := (p_t)_{|I\times J}$. Under Assumption \ref{condition_m}, there exists a constant $\mathfrak c_{\ref{empirical_risk_bound}} > 0$, not depending on $\mathbf m$ and $N$, such that for every $t\in [0,T]$,
\begin{equation}\label{empirical_risk_bound_1}
\mathbb E(\|\widehat p_{\mathbf m,t} - p_{I\times J,t}\|_{N}^{2})
\leqslant
\min_{\tau\in\mathcal S_{\bf m}}\|\tau - p_{I\times J,t}\|_{f}^{2} +
\frac{2m_1\mathfrak L_{\psi}(m_2)}{N} +
\frac{\mathfrak c_{\ref{empirical_risk_bound}}}{N}.
\end{equation}
\end{theorem}
\noindent
We emphasize that the risk bound in Theorem \ref{empirical_risk_bound} is sharp since the constant in front of the bias term $\min_{\tau\in\mathcal S_{\bf m}}\|\tau - p_{I\times J,t}\|_{f}^{2}$ in Inequality (\ref{empirical_risk_bound_1}) is $1$, and the constant 2 in the variance term $2m_1\mathfrak L_{\psi}(m_2)/N$ may be $1 +\varepsilon$ for $\varepsilon > 0$, up to additional technicalities.
\\
\\
\textbf{Sketch of proof.} Let us provide a sketch of the proof of Theorem \ref{empirical_risk_bound}, which is detailed in Subsection \ref{section_proof_empirical_risk_bound}. For every function $h$ from $\mathbb R^2$ into $\mathbb R$, let
\begin{displaymath}
\widehat\Pi_{\bf m}(h) :=
\arg\min_{\tau\in\mathcal S_{\bf m}}
\|\tau - h\|_{N}^{2}
\quad\textrm{be its empirical projection on $\mathcal S_{\bf m}$.}
\end{displaymath}
Then,
\begin{eqnarray*}
 \mathbb E(\|\widehat p_{\mathbf m,t} - p_{I\times J,t}\|_{N}^{2})
 & = &
 \mathbb E\left(\min_{\tau\in\mathcal S_{\bf m}}
 \|\tau - p_{I\times J,t}\|_{N}^{2}\right) +
 \mathbb E(\|\widehat p_{\mathbf m,t} -\widehat\Pi_{\bf m}(p_t)\|_{N}^{2}\\
 & \leqslant &
 \min_{\tau\in\mathcal S_{\bf m}}\|\tau - p_{I\times J,t}\|_{f}^{2} +
 \mathbb E(\|\widehat p_{\mathbf m,t} -\widehat\Pi_{\bf m}(p_t)\|_{N}^{2}).
\end{eqnarray*}
The first term is the expected squared bias term, and
\begin{eqnarray*}
 \mathbb E(\|\widehat p_{\mathbf m,t} -\widehat\Pi_{\bf m}(p_t)\|_{N}^{2}) & = &
 \mathbb E(\|\widehat p_{\mathbf m,t} -\widehat\Pi_{\bf m}(p_t)\|_{N}^{2}\mathbf 1_{\Omega_{m_1}})\\
 & &
 \hspace{2cm} +
 \mathbb E(\|\widehat p_{\mathbf m,t} -
 \widehat\Pi_{\bf m}(p_t)\|_{N}^{2}\mathbf 1_{\Omega_{m_1}^{c}}) =: A + B,
\end{eqnarray*}
where $\Omega_{m_1}$ is the event on which, for every $\mathbf x\in\mathbb R^{m_1}$, $\mathbf x^*\widehat\Psi_{m_1}\mathbf x\leqslant 2\mathbf x^*\Psi_{m_1}\mathbf x$. This means that on $\Omega_{m_1}$, $\widehat\Psi_{m_1}$ can be replaced by its expectation up to a factor two in controls. In a first step, $A$ is controlled by $m_1\mathfrak L_{\psi}(m_2)/N$, the main variance contribution, by using that
\begin{equation}\label{empirical_risk_bound_3}
\|\widehat p_{\mathbf m,t} -\widehat\Pi_{\bf m}(p_t)\|_{N}^{2} =
\sup_{\tau\in\mathcal S_{\bf m} :\|\tau\|_N = 1}\left[
\frac{1}{NT}\sum_{i = 1}^{N}\int_{0}^{T}\left(
\tau(X_{s}^{i},X_{s + t}^{i}) -\int_{-\infty}^{\infty}\tau(X_{s}^{i},y)p_t(X_{s}^{i},y)dy\right)ds\right]^2,
\end{equation}
and by bounding its expectation thanks to the definition of $\Omega_{m_1}$. Indeed, on $\Omega_{m_1}$, $\|\tau\|_{N}^{2}\leqslant 1$ implies that $\|\tau\|_{f}^{2}\leqslant 2$, and then the random unit ball on which the supremum is taken can be replaced by a deterministic one. In a second step, we prove that $B = O(1/N)$ thanks to Comte and Genon-Catalot \cite{CGC20}, Lemma 6.1, which says that $\mathbb P(\Omega_{m_1}^{c})\lesssim 1/N^p$.
\\
\\
{\bf Remark.} Let us discuss the order of the variance term in Inequality (\ref{empirical_risk_bound_1}) for the bases presented in Section \ref{section_bases}. Obviously, the choice of $\varphi$ does not matter as it leads in any case to the factor $m_1$. So, we discuss the choice of $\psi$, using the results of Section \ref{section_bases}. First, for $\psi$ chosen as splines, wavelets or trigonometric bases, $\mathfrak L_{\psi}(m_2)\lesssim m_2$, and then the variance term in the risk bound on $\widehat p_{\mathbf m,t}$ is of order $m_1m_2/N$ as for the usual projection estimator of a 2-dimensional density function. Now, for Legendre's basis, $\mathfrak L_{\psi}(m_2)\lesssim m_{2}^{2}$, leading to a variance term of order $m_1m_{2}^{2}/N$ in Inequality (\ref{empirical_risk_bound_1}). Finally, for the Hermite basis, it is worth noting that, by (\ref{L_Hermite_basis}), the variance term in Inequality (\ref{empirical_risk_bound_1}) is of order $m_1\sqrt m_2/N$.
%

% Subsection : Risk bound with respect to the f-weighted norm, on a truncated version of the projection least squares estimator.

%
\subsection{Risk bound with respect to the $f$-weighted norm, on a truncated version of the projection least squares estimator}\label{section_f_weighted_risk_bound_truncated}
In Theorem \ref{empirical_risk_bound}, the risk of $\widehat p_{\mathbf m,t}$ is related to the empirical norm $\|.\|_N$. Then, a result for a risk related to the theoretical counterpart $\|.\|_f$ of $\|.\|_N$ is of interest, especially because the bias term in Inequality (\ref{empirical_risk_bound_1}) involves $\|.\|_f$. In order to establish such a result from the "discrete" (empirical) one, an empirical version of the constraint (\ref{stability_condition}) is required to ensure that the eigenvalues of $\widehat\Psi_{m_1}$ are not too close to $0$. Thus, this subsection deals with a risk bound, with respect to the $f$-weighted norm $\|.\|_f$, on the following truncated version of our estimator $\widehat p_{\mathbf m,t}$:
\begin{displaymath}
\widetilde p_{\mathbf m,t} :=\widehat p_{\mathbf m,t}\mathbf 1_{\Lambda_{m_1}},
\end{displaymath}
where
\begin{equation}\label{event_Lambda}
\Lambda_{m_1} :=
\left\{\mathfrak L_{\varphi}(m_1)(\|\widehat\Psi_{m_1}^{-1}\|_{\rm op}\vee 1)
\leqslant\mathfrak c_{\Lambda}\frac{NT}{\log(NT)}
\right\}.
\end{equation}
Under (\ref{stability_condition}), the probability of $\Lambda_{m_1}$ is near of one. Moreover, as expected, on the event $\Lambda_{m_1}$, $\widehat\Psi_{m_1}$ is invertible because
\begin{displaymath}
\inf\{{\rm sp}(\widehat\Psi_{m_1})\}
\geqslant\frac{\mathfrak L_{\varphi}(m_1)}{\mathfrak c_{\Lambda}}\cdot\frac{\log(NT)}{NT},
\end{displaymath}
where ${\rm sp}(\widehat\Psi_{m_1})$ is the spectrum of the random matrix $\widehat\Psi_{m_1}$, and then $\widetilde p_{\mathbf m,t}$ is well-defined.
%

% Theorem : f-weighted risk bound on the truncated estimator.

%
\begin{theorem}\label{f_weighted_risk_bound_truncated}
Under Assumption \ref{condition_m},
\begin{enumerate}
 \item There exists a constant $\mathfrak c_{\ref{f_weighted_risk_bound_truncated},1} > 0$, not depending on $\mathbf m$ and $N$, such that for every $t\in (0,T]$,
 \begin{equation}\label{f_weighted_risk_bound_truncated_1}
 \mathbb E(\|\widetilde p_{\mathbf m,t} - p_{I\times J,t}\|_{f}^{2})
 \leqslant
 9\min_{\tau\in\mathcal S_{\bf m}}\|\tau - p_{I\times J,t}\|_{f}^{2} +
 \frac{8m_1\mathfrak L_{\psi}(m_2)}{N} +
 \frac{\mathfrak c_{\ref{f_weighted_risk_bound_truncated},1}(1 + R_f(t))}{N}
 \end{equation}
 with $R_f(t) = R(t) +\|p_{I\times J,t}\|_{f}^{2}$ and
 \begin{displaymath}
 R(t) =
 \frac{1}{T}\mathbb E\left[\left(\int_{0}^{T}\int_{-\infty}^{\infty}
 p_t(X_s,y)^2dyds\right)^2\right]^{\frac{1}{2}}.
 \end{displaymath}
 \item There exists a constant $\mathfrak c_{\ref{f_weighted_risk_bound_truncated},2} > 0$, not depending on $\mathbf m$ and $N$, such that
 \begin{equation}\label{f_weighted_risk_bound_truncated_2}
 \frac{1}{T}\int_{0}^{T}\mathbb E(\|\widetilde p_{\mathbf m,t} - p_{I\times J,t}\|_{f}^{2})dt
 \leqslant
 9\min_{\tau\in\mathcal S_{\bf m}}\left\{\frac{1}{T}\int_{0}^{T}
 \|\tau - p_{I\times J,t}\|_{f}^{2}dt\right\} +
 \frac{8m_1\mathfrak L_{\psi}(m_2)}{N} +
 \frac{\mathfrak c_{\ref{f_weighted_risk_bound_truncated},2}}{N}.
 \end{equation}
\end{enumerate}
\end{theorem}
\noindent
\textbf{Sketch of proof.} Let us provide a sketch of the proof of Theorem \ref{f_weighted_risk_bound_truncated}, which is detailed in Subsection \ref{section_proof_f_weighted_risk_bound_truncated}. In a first step, thanks to the definitions of both $\Omega_{m_1}$ and the orthogonal projection of $p_{I\times J,t}$ onto $\mathcal S_{\bf m}$ for the inner product $\langle .,.\rangle_f$, we show that
\begin{equation}\label{f_weighted_risk_bound_truncated_3}
\mathbb E(\|\widetilde p_{\mathbf m,t} - p_{I\times J,t}\|_{f}^{2})
\lesssim
\min_{\tau\in\mathcal S_{\bf m}}
\|\tau - p_{I\times J,t}\|_{f}^{2} + A + B + C,
\end{equation}
where
\begin{eqnarray*}
 A & := &
 \mathbb E(\|\widehat p_{\mathbf m,t} - p_{I\times J,t}\|_{N}^{2}
 \mathbf 1_{\Omega_{m_1}\cap\Lambda_{m_1}}),\\
 B & := &
 \mathbb E(\|p_{I\times J,t}\|_{N}^{2}
 \mathbf 1_{\Omega_{m_1}\cap\Lambda_{m_1}^{c}})\quad {\rm and}\quad
 C :=\mathbb E(\|\widetilde p_{\mathbf m,t} - p_{I\times J,t}\|_{f}^{2}
 \mathbf 1_{\Omega_{m_1}^{c}}).
\end{eqnarray*}
The (first) bias term in the right-hand side of Inequality (\ref{f_weighted_risk_bound_truncated_3}) remains the same as in Inequality (\ref{empirical_risk_bound_1}). In the same way, $A$ is of order $m_1\mathfrak L_{\psi}(m_2)/N + 1/N$ as in Theorem \ref{empirical_risk_bound}. Since $\mathbb P(\Lambda_{m_1}^{c})\lesssim 1/N^p$ by Comte and Genon-Catalot \cite{CGC20}, Lemma 6.1, we get that $B\lesssim R(t)/N$. Lastly, $C$ is controlled by $(1 +\|p_{I\times J,t}\|_{f}^{2})/N$ by using that
\begin{displaymath}
\|\widehat p_{\mathbf m,t}\|_{f}^{2} =
{\rm trace}(\Psi_{m_1}\widehat\Theta_{\mathbf m,t}\widehat\Theta_{\mathbf m,t}^{*})
\quad {\rm and}\quad
\mathbb P(\Omega_{m_1}^{c})\lesssim\frac{1}{N^p}.
\end{displaymath}
Inequality (\ref{f_weighted_risk_bound_truncated_2}) is deduced from Inequalities (\ref{Gaussian_bound_density}) and (\ref{f_weighted_risk_bound_truncated_1}).
\\
\\
{\bf Remarks:}
\begin{enumerate}
 \item When $I$ is a compact interval, the norms $\|.\|$ and $\|.\|_f$ are equivalent by Proposition \ref{properties_transition}.(1,2). In that case, one can deduce a risk bound on $\widetilde p_{\mathbf m,t}$, involving $\|.\|$ and not $\|.\|_f$, from Theorem \ref{f_weighted_risk_bound_truncated}. This is no longer true when $I$ is a non-compact interval, and the risk needs to be evaluated via the $f$-weighted norm $\|.\|_f$.
 \item By Jensen's inequality, Inequality (\ref{Gaussian_lower_bound_density}) and the change of variable formula,
 \begin{eqnarray*}
  R(t) & \geqslant &
  \frac{1}{T}\int_{0}^{T}\int_{-\infty}^{\infty}\mathbb E(p_t(X_s,y)^2)dyds =
  \int_{-\infty}^{\infty}\int_{-\infty}^{\infty}p_t(x,y)^2f(x)dxdy\\
  & \geqslant &
  \frac{\underline{\mathfrak c}_{T}^{2}}{t}\int_{-\infty}^{\infty}f(x)
  \int_{-\infty}^{\infty}\exp\left(-2\underline{\mathfrak m}_T\frac{(y -\theta_t(x))^2}{t}\right)dydx\\
  & = &
  \frac{\underline{\mathfrak c}_{T}^{2}}{t}
  \int_{-\infty}^{\infty}f(x)dx
  \int_{-\infty}^{\infty}\exp\left[-2\underline{\mathfrak m}_T\left(\frac{y}{t^{1/2}}\right)^2\right]dy\\
  & = &
  \frac{\underline{\mathfrak c}_{T}^{2}}{t^{1/2}}
  \int_{-\infty}^{\infty}e^{-2\underline{\mathfrak m}_Ty^2}dy =
  \frac{2\underline{\mathfrak c}_{T}^{2}\sqrt\pi}{\sqrt{2\underline{\mathfrak m}_Tt}}
  \xrightarrow[t\rightarrow 0^+]{} +\infty.
 \end{eqnarray*}
 Therefore,
 \begin{displaymath}
 \sup_{t\in (0,T]}R(t) = +\infty,
 \end{displaymath}
 and this is the reason why Inequality (\ref{f_weighted_risk_bound_truncated_2}) is also provided. Otherwise, only times $t\in [t_0,T]$, for some fixed $t_0\in (0,T)$, must be considered.
\end{enumerate}
%

% Subsection : Rates in the anisotropic Sobolev-Hermite spaces.

%
\subsection{Rates in the anisotropic Sobolev-Hermite spaces}
In order to control the bias term in Theorems \ref{empirical_risk_bound} and \ref{f_weighted_risk_bound_truncated}, we assume that $p_t$ belongs to a Sobolev-Hermite space. In dimension $d = 1$, these function spaces have been introduced by Bongioanni and Torrea \cite{BT06}. The connection between Hermite's coefficients and functional regularity was established later (see Belomestny {\it et al.} \cite{BCGC19}) and are summarized in Comte and Genon-Catalot \cite{CGC18}, Appendix A. The definition of these spaces can be extended on $A :=\mathbb R^d$ for any $d\in\mathbb N^*$ (see Comte and Lacour \cite{CL23}, Section 2.3).
\\
\\
{\bf Notations:}
\begin{itemize}
 \item For every $\mathbf k = (k_1,\dots,k_d)\in\mathbb N^d$ and $\mathbf s = (s_1,\dots,s_d)\in (0,\infty)^d$, $\mathbf k^{\bf s} := k_{1}^{s_1}\times\dots\times k_{d}^{s_d}$.
 \item For every $g\in\mathbb L^2(A)$ and $\mathbf k = (k_1,\dots,k_d)\in\mathbb N^d$, $a_{\bf k}(g) :=\langle g,h_{k_1}\otimes\dots\otimes h_{k_d}\rangle$.
\end{itemize}
Throughout this subsection, both $\mathcal S_{\varphi,m_1}$ and $\mathcal S_{\psi,m_2}$ are generated by the Hermite basis, and then $I = J =\mathbb R$. First, let us recall the definition of the Sobolev-Hermite ellipsoid on $A :=\mathbb R^d$, of order $\mathbf s = (s_1,\dots,s_d)\in (0,\infty)^d$ and of radius $L > 0$.
%

% Definition : Sobolev-Hermite ellipsoids.

%
\begin{definition}\label{Sobolev_Hermite_ellipsoids} (Sobolev-Hermite ellispoids) The Sobolev-Hermite ellipsoid $W_{\mathbf s}^{(d)}(A,L)$ of order $\mathbf s$ and of radius $L$ is defined by
\begin{displaymath}
W_{\mathbf s}^{(d)}(A,L) :=
\left\{g\in\mathbb L^2(A) :
\sum_{\mathbf k\in\mathbb N^d}
a_{\bf k}(g)^2\mathbf k^{\bf s}\leqslant L\right\}.
\end{displaymath}
\end{definition}
\noindent
Now, by assuming that $p_t$ belongs to $W^{(2)}_{\bf s}(A,L)$, the bias term in Theorems \ref{empirical_risk_bound} and \ref{f_weighted_risk_bound_truncated} decreases to $0$ with polynomial rate. Indeed, noting $p_{\mathbf m,t}$ the orthogonal projection of $p_t$ on $\mathcal S_{\bf m}$,
\begin{eqnarray*}
 \|p_t - p_{\mathbf m,t}\|^2 & = &
 \sum_{\mathbf k\in\mathbb N^2 :\exists q\in\{1,2\}, k_q\geqslant m_q}a_{\bf k}(p_t)^2\\
 & \leqslant &
 \sum_{q = 1}^{2}\sum_{\mathbf k\in {\mathbb N}^2 : k_q\geqslant m_q}
 a_{\bf k}(p_t)^2k_{q}^{s_q}k_{q}^{-s_q}\leqslant
 L(m_{1}^{-s_1} + m_{2}^{-s_2})
\end{eqnarray*}
and then, since $p_{I\times J,t}=p_t$ and $f(\cdot)\leqslant 2\mathfrak c_TT^{-1/2}$,
\begin{displaymath}
\min_{\tau\in\mathcal S_{\bf m}}\|\tau - p_{I\times J,t}\|_{f}^{2}\leqslant
\frac{2\mathfrak c_TL}{\sqrt T}\|p_t - p_{\mathbf m,t}\|^2\leqslant
\frac{2\mathfrak c_TL}{\sqrt T}(m_{1}^{-s_1} + m_{2}^{-s_2}).
\end{displaymath}
Thanks to this control of the bias term, and since the variance term in Theorem \ref{empirical_risk_bound} is of order $m_1\sqrt m_2/N$ when $\mathcal S_{\psi,m_2}$ is generated by the Hermite basis (see Subsection \ref{section_empirical_risk_bound}), one can establish the following proposition.
%

% Proposition : Empirical risk bound (Sobolev-Hermite).

%
\begin{proposition}\label{empirical_risk_bound_Sobolev_Hermite}
Assume that $p_t$ belongs to $W_{\mathbf s}^{(2)}(A,L)$, and consider $\mathbf m^* = (m_{1}^{*},m_{2}^{*})$ with
\begin{displaymath}
m_{1}^{*}\propto N^{\frac{s_2}{s_1s_2 + s_1/2 + s_2}}
\quad\textrm{ and }\quad
m_{2}^{*}\propto N^{\frac{s_1}{s_1s_2 + s_1/2 + s_2}}.
\end{displaymath}
Then,
\begin{displaymath}
\mathbb E(\|\widehat p_{\mathbf m^*,t} - p_t\|_{N}^{2}) =
O\left(N^{-\frac{1}{1 +\frac{1}{s_1} +\frac{1}{2s_2}}}\right),
\end{displaymath}
provided that $m_{1}^{*}$ satisfies the stability condition (\ref{stability_condition}) in Assumption \ref{condition_m}.
\end{proposition} 
\noindent
The proof of Proposition \ref{empirical_risk_bound_Sobolev_Hermite} follows the same line as Comte and Lacour \cite{CL23}, Proposition 2, and is omitted here. We mention that, for conditional density estimation, Comte and Lacour \cite{CL23} also prove a lower bound for this rate (see their Theorem 2 in Section 3.5), under some additional conditions. Note that, for instance, other rates may be obtained on Sobolev-Fourier spaces associated with the trigonometric basis.
%

% Section : Model selection.

%
\section{Model selection}\label{section_model_selection}
Clearly the values of $m_{1}^{*}$ and $m_{2}^{*}$ depend on $s_1$ and $s_2$, the regularity indexes of the unknown function to estimate. Therefore, these results have theoretical interest to show the consistency and evaluate the rate of the estimators, but cannot be used in practice. In this section, we study a model selection method leading to a data-driven choice of $m_1$ and $m_2$. 
%

% Subsection : General case.

%
\subsection{General case}
In order to introduce an appropriate model selection criterion, throughout this section, the $\varphi_j$'s and the $\psi_{\ell}$'s fulfill the following additional assumption.
%

% Assumption : Assumption on the projection spaces.

%
\begin{assumption}\label{assumption_projection_spaces}
The $\varphi_j$'s and the $\psi_{\ell}$'s fulfill the three following conditions:
\begin{enumerate}
 \item For every $m_1,M_1\in\{1,\dots,N_T\}$, if $M_1 > m_1$, then $\mathcal S_{\varphi,m_1}\subset\mathcal S_{\varphi,M_1}$.
 \item For every $m_2,M_2\in\{1,\dots,N_T\}$, if $M_2 > m_2$, then $\mathcal S_{\psi,m_2}\subset\mathcal S_{\psi,M_2}$.
 \item There exists a constant $\mathfrak c_{\varphi}\geqslant 1$, not depending on $N$, such that
 \begin{displaymath}
 \mathfrak L_{\varphi}(m_1)\leqslant\mathfrak c_{\varphi}^{2}m_1
 \textrm{ $;$ }\forall m_1\in\{1,\dots,N_T\}.
 \end{displaymath}
\end{enumerate}
\end{assumption}
\noindent
Note that the first two conditions above mean that the two univariate collections of models are nested, which of course does not imply that the product spaces are. For instance, the compactly supported trigonometric basis, and both the non-compactly supported Laguerre and Hermite bases, fulfill Assumption \ref{assumption_projection_spaces}.
\\
\\
Now, we define a data-driven criterion allowing to select a model $\widehat{\bf m}$ making a squared-bias/variance compromise. Such a criterion is a function, to minimize with respect to $\mathbf m$ taken in a finite collection, estimating the $\mathbb L^2$-risk as the sum of the estimated squared bias of $\widehat p_{\mathbf m,t}$ penalized by a term having the order of its variance term. On the one hand, the squared bias of $\widehat p_{\mathbf m,t}$ is estimated by $\gamma_N(\widehat p_{\mathbf m,t}) = -\|\widehat p_{\mathbf m,t}\|_{N}^{2}$. On the other hand, the order of the variance should be  estimated by its bound $m_1\mathfrak L_{\psi}(m_2)/N$, up to a multiplicative constant to calibrate.  But in the general case, we need to take a penalty of order $m_1\mathfrak L_{\psi}(m_2)\log(N)/N$, implying a slightly degraded rate of the adaptive estimator. Precisely, we consider
\begin{equation}\label{model_selection_criterion}
\widehat{\bf m} =\widehat{\bf m}(t) :=
\arg\min_{\mathbf m\in\widehat{\mathcal M}_N}
\{-\|\widehat p_{\mathbf m,t}\|_{N}^{2} + 2\kappa{\rm pen}(\mathbf m)\},
\end{equation}
where $\kappa\geqslant\kappa_0$, $\kappa_0 > 0$ is defined later,
\begin{equation}\label{penalty}
{\rm pen}(\mathbf m) :=
(1 +\log(N))\frac{m_1\mathfrak L_{\psi}(m_2)}{N}
\textrm{ $;$ }
\forall\mathbf m = (m_1,m_2)\in\{1,\dots,N_T\}^2,
\end{equation}
and $\widehat{\mathcal M}_N :=\mathcal U_N\cap (\widehat{\mathcal V}_N\times\mathcal N)$ with $\mathcal N =\{1,\dots,N\wedge N_T\}$,
\begin{eqnarray*}
 \mathcal U_N & = &
 \{(m_1,m_2)\in\mathcal N^2 : m_1\mathfrak L_{\psi}(m_2)\leqslant N\},\\
 & &
 \hspace{1cm}
 \widehat{\mathcal V}_N =
 \left\{m_1\in\mathcal N :
 \mathfrak c_{\varphi}^{2}m_1(\|\widehat\Psi_{m_1}^{-1}\|_{\rm op}^{2}\vee 1)
 \leqslant\mathfrak d\frac{NT}{\log(NT)}\right\}
\end{eqnarray*}
and
\begin{displaymath}
\mathfrak d =
\min\left(\frac{1}{8\mathfrak c_{\varphi}^{2}T(\|f\|_{\infty} +
(3\mathfrak c_{\varphi})^{-1}\sqrt{\mathfrak c_{\Lambda}/8})(1 + p)},
\frac{\mathfrak c_{\Lambda}}{8}\right)
\quad (\mathfrak c_{\Lambda} =\frac{1 -\log(2)}{(1 + p)T}).
\end{displaymath}
Note that since $\mathfrak c_{\varphi}$, $\mathfrak c_{\Lambda}$, $p$, $T$ and $\|f\|_{\infty}$ are all positive, then $\mathfrak d$ is well-defined. Note also that for every $\mathbf m = (m_1,m_2)\in\mathcal U_N$, $m_2$ fulfills the first part (\ref{weak_limit_psi}) of Assumption \ref{condition_m} with $q = 1$ (or larger), and then we take $p = 11$ (recall that $p = 2(q + 4) + 1$) in this section. Consider also the theoretical counterpart $\mathcal M_N :=\mathcal U_N\cap (\mathcal V_N\times\mathcal N)$ of $\widehat{\mathcal M}_N$, where
\begin{displaymath}
\mathcal V_N :=
\left\{m_1\in\mathcal N :
\mathfrak c_{\varphi}^{2}m_1(\|\Psi_{m_1}^{-1}\|_{\rm op}^{2}\vee 1)
\leqslant\frac{\mathfrak d}{4}\cdot\frac{NT}{\log(NT)}\right\}.
\end{displaymath}
The following theorem provides a risk bound on the adaptive estimator $\widehat p_{\widehat{\bf m},t}$.
%

% Theorem : Risk bound on the adaptive estimator.

%
\begin{theorem}\label{risk_bound_adaptive_estimator}
Under Assumption \ref{assumption_projection_spaces}, for $\kappa_0 = 44a$ and $a\geqslant (2\cdot 84\sqrt{\mathfrak dT})^2/2$,
\begin{enumerate}
 \item There exists a constant $\mathfrak c_{\ref{risk_bound_adaptive_estimator},1} > 0$, not depending on $N$, such that for every $t\in (0,T]$,
 \begin{displaymath}
 \mathbb E(\|\widehat p_{\widehat{\bf m},t} - p_{I\times J,t}\|_{N}^{2})
 \leqslant
 6\min_{\mathbf m\in\mathcal M_N}\{\mathbb E(\|\widehat p_{\mathbf m,t} - p_{I\times J,t}\|_{N}^{2}) +
 \kappa{\rm pen}(\mathbf m)\} +
 \frac{\mathfrak c_{\ref{risk_bound_adaptive_estimator},1}(1 + R(t))}{N}.
 \end{displaymath}
 \item There exists a constant $\mathfrak c_{\ref{risk_bound_adaptive_estimator},2} > 0$, not depending on $N$, such that
 \begin{displaymath}
 \frac{1}{T}\int_{0}^{T}\mathbb E(\|\widehat p_{\widehat{\bf m}(t),t} - p_{I\times J,t}\|_{N}^{2})dt
 \leqslant
 6\min_{\mathbf m\in\mathcal M_N}\left\{
 \frac{1}{T}\int_{0}^{T}\mathbb E(\|\widehat p_{\mathbf m,t} - p_{I\times J,t}\|_{N}^{2})dt +
 \kappa{\rm pen}(\mathbf m)\right\} +
 \frac{\mathfrak c_{\ref{risk_bound_adaptive_estimator},2}}{N}.
 \end{displaymath}
\end{enumerate}
\end{theorem}
\noindent
The first result in Theorem \ref{risk_bound_adaptive_estimator} means that the final estimator $\widehat p_{\widehat{\bf m},t}$ makes automatically (up to a multiplicative constant which may be taken equal to $6+2\kappa$) the bias-variance tradeoff by keeping in mind Inequality (\ref{empirical_risk_bound_1}) which provides a risk bound on $\mathbb E(\|\widehat p_{\mathbf m,t} - p_{I\times J,t}\|_{N}^{2})$ for every $\mathbf m\in\mathcal M_N$.
\\
\\
\textbf{Sketch of proof.} Let us provide a sketch of the proof of Theorem \ref{risk_bound_adaptive_estimator}, which is detailed in Subsection \ref{section_proof_risk_bound_adaptive_estimator}. In a first step, by using that $\widehat\Pi_{\mathbf m_1}(\widehat p_{\mathbf m_2,t}) =\widehat p_{\mathbf m_1,t}$ for every $\mathbf m_1,\mathbf m_2\in\mathcal N^2$ satisfying $\mathcal S_{\mathbf m_1}\subset\mathcal S_{\mathbf m_2}$, we show that for every $\mathbf m\in\widehat{\mathcal M}_N$,
\begin{equation}\label{risk_bound_adaptive_estimator_6}
\|\widehat p_{\widehat{\bf m},t} - p_{I\times J,t}\|_{N}^{2}\leqslant
6\|\widehat p_{\mathbf m,t} - p_{I\times J,t}\|_{N}^{2} +
4\kappa{\rm pen}(\mathbf m) + 11\left(
\|\widehat p_{\widehat{\bf m},t} -\widehat{\Pi}_{\widehat{\bf m}}(p_t)\|_{N}^{2} -
\frac{2}{11}\kappa{\rm pen}(\widehat{\bf m})\right)_+.
\end{equation}
Now, consider the event $\Xi_N :=\{\mathcal M_N\subset\widehat{\mathcal M}_N\subset\mathfrak M_N\}$ with $\mathfrak M_N =\mathcal U_N\cap (\mathfrak V_N\times\mathcal N)$ and
\begin{displaymath}
\mathfrak V_N =
\left\{m_1\in\mathcal N :
\mathfrak c_{\varphi}^{2}m_1(\|\Psi_{m_1}^{-1}\|_{\rm op}^{2}\vee 1)
\leqslant 4\mathfrak d\frac{NT}{\log(NT)}\right\},
\end{displaymath}
and recall that $\Omega_{m_1}$ ($m_1\in\mathcal N$) is an event on which the empirical norm $\|.\|_{N,1}$ and its theoretical counterpart $\|.\|_{f,1}$ are equivalent for functions in $\mathcal S_{\varphi,m_1}$. By Comte and Genon-Catalot \cite{CGC20}, Lemma 6.1 and Inequality (6.17), the events $\Xi_N$ and $\Omega_N :=\cap_{m_1\in\mathfrak V_N}\Omega_{m_1}$ have probability near of 1:
\begin{equation}\label{risk_bound_adaptive_estimator_7}
\mathbb P(\Xi_{N}^{c})\lesssim\frac{1}{N^{p - 1}}
\quad {\rm and}\quad
\mathbb P(\Omega_{N}^{c})\lesssim\frac{1}{N^{p - 1}}.
\end{equation}
Consider also the empirical process $\nu_N$ defined by
\begin{displaymath}
\nu_N(\tau) :=
\frac{1}{NT}\sum_{i = 1}^{N}\int_{0}^{T}\left(
\tau(X_{s}^{i},X_{s + t}^{i}) -\int_{-\infty}^{\infty}\tau(X_{s}^{i},y)p_t(X_{s}^{i},y)dy\right)ds
\textrm{ $;$ }\forall\tau\in\mathcal S_{\bf m}.
\end{displaymath}
In a second step, thanks to Equality (\ref{empirical_risk_bound_3}), we deduce from (\ref{risk_bound_adaptive_estimator_7}) and Inequality (\ref{risk_bound_adaptive_estimator_6}) that
\begin{eqnarray*}
 \mathbb E(\|\widehat p_{\widehat{\bf m},t} - p_{I\times J,t}\|_{N}^{2})
 & \lesssim &
 \min_{\mathbf m\in\mathcal M_N}\{
 \mathbb E(\|\widehat p_{\mathbf m,t} - p_{I\times J,t}\|_{N}^{2}) +
 \kappa{\rm pen}(\mathbf m)\}\\
 & &
 \hspace{0.75cm} +
 \sum_{\mathbf m\in\mathfrak M_N}
 \mathbb E\left[\left(\sup_{\tau\in\mathcal S_{\bf m} :\|\tau\|_f\leqslant 2}\nu_N(\tau)^2
 -\frac{2}{11}\kappa{\rm pen}(\mathbf m)\right)_+\mathbf 1_{\Xi_N\cap\Omega_N}\right] +
 \frac{1 + R(t)}{N}.
\end{eqnarray*}
In a third step, we prove that
\begin{displaymath}
\sum_{\mathbf m\in\mathfrak M_N}
\mathbb E\left[\left(\sup_{\tau\in\mathcal S_{\bf m} :\|\tau\|_f\leqslant 2}\nu_N(\tau)^2
-\frac{2}{11}\kappa{\rm pen}(\mathbf m)\right)_+\mathbf 1_{\Xi_N\cap\Omega_N}\right]
\lesssim\frac{1}{N}
\end{displaymath}
thanks to Klein and Rio's extension of Talagrand's inequality (see \cite{KR05}, Theorem 1.2). In a fourth step, we deduce Theorem \ref{risk_bound_adaptive_estimator}.(2) from Theorem \ref{risk_bound_adaptive_estimator}.(1) and Inequality (\ref{Gaussian_bound_density}).
%

% Subsection : Two special cases: t > t_0 > 0 and compactly supported bases.

%
\subsection{Two special cases: $t > t_0 > 0$ and compactly supported bases}
This subsection deals with two interesting special cases, where the extra $\log(N)$ term in the penalty given by (\ref{penalty}), can be avoided. First, let us consider $t\in [t_0,T]$ with a fixed $t_0 > 0$. This condition on $t$ leads to
\begin{displaymath}
\sup_{(x,y)\in I\times J}p_t(x,y)\leqslant p_0 :=
\overline{\mathfrak c}_Tt_{0}^{-\frac{1}{2}}
\quad\textrm{by Inequality (\ref{Gaussian_bound_density}).}
\end{displaymath}
The following model selection criterion, simpler than (\ref{model_selection_criterion}), may be considered:
\begin{equation}\label{model_selection_criterion_bis}
\widetilde{\bf m} =\widetilde{\bf m}(t) :=
\arg\min_{\mathbf m\in\widehat{\mathcal M}_N}
\{-\|\widehat p_{\mathbf m,t}\|_{N}^{2} + 2\kappa_b{\rm pen}_b(\mathbf m)\},
\end{equation}
where $\kappa_b\geqslant\kappa_{b,0}$, $\kappa_{b,0} > 0$ is defined later, and
\begin{displaymath}
{\rm pen}_b(\mathbf m) :=
\frac{m_1\mathfrak L_{\psi}(m_2)}{N}
\textrm{ $;$ }
\forall\mathbf m = (m_1,m_2)\in\{1,\dots,N_T\}^2.
\end{displaymath}
In the sequel, the map $(m_1,m_2)\mapsto m_1\mathfrak L_{\psi}(m_2)$ fulfills the following additional but usual condition.
%

% Assumption : Condition on \mathbf m (bis).

%
\begin{assumption}\label{condition_m_bis}
For every $\xi > 0$, there exists $S(\xi) > 0$, not depending on $N$, such that
\begin{displaymath}
\sum_{1\leqslant m_1,m_2\leqslant N}
\exp(-\xi m_1\mathfrak L_{\psi}(m_2))\leqslant S(\xi) <\infty.
\end{displaymath}
\end{assumption}
\noindent
Assumption \ref{condition_m_bis} means that the complexity of the models collection is not too high. Note that if $\mathfrak L_{\psi}(m_2) = 1\vee (\mathfrak c_{\psi}m_2)$, or even if $\mathfrak L_{\psi}(m_2) = 1\vee (\mathfrak c_{\psi}\sqrt m_2)$, then $(m_1,m_2)\mapsto m_1\mathfrak L_{\psi}(m_2)$ fulfills Assumption \ref{condition_m_bis} because
\begin{displaymath}
\sum_{m_1 = 1}^{N}\exp(-\xi\mathfrak L_{\psi}(m_2))^{m_1}
\leqslant\frac{e^{-\xi\mathfrak L_{\psi}(m_2)}}{1 - e^{-\xi}}
\textrm{ $;$ }\forall m_2\in\mathbb N^*.
\end{displaymath}
So, for instance, $(m_1,m_2)\mapsto m_1\mathfrak L_{\psi}(m_2)$ fulfills Assumption \ref{condition_m_bis} when $\mathcal S_{\psi,m_2}$ is generated by the trigonometric or the Hermite basis.
%

% Proposition : Risk bound on the adaptive estimator (bis).

%
\begin{proposition}\label{risk_bound_adaptive_estimator_bis}
Under Assumptions \ref{assumption_projection_spaces} and \ref{condition_m_bis}, for $\kappa_{b,0} = 16.5$, there exists a constant $\mathfrak c_{\ref{risk_bound_adaptive_estimator_bis}} > 0$, not depending on $N$, such that for every $t\in [t_0,T]$ with $t_0 > 0$,
\begin{displaymath}
\mathbb E(\|\widehat p_{\widetilde{\bf m},t} - p_{I\times J,t}\|_{N}^{2})
\leqslant
6\min_{\mathbf m\in\mathcal M_N}\{\mathbb E(\|\widehat p_{\mathbf m,t} - p_{I\times J,t}\|_{N}^{2}) +
\kappa_b{\rm pen}_b(\mathbf m)\} +
\frac{\mathfrak c_{\ref{risk_bound_adaptive_estimator_bis}}(1 + t_0^{-1/2})}{N}.
\end{displaymath}
\end{proposition}
\noindent
Now, let us briefly present the second interesting special case by assuming that $I$ is a compact interval. As already established in Section \ref{section_assumptions}, by Proposition \ref{properties_transition}.(2), there exists $\underline{\mathfrak m} > 0$ such that $f(\cdot)\geqslant\underline{\mathfrak m}$ on $I$. This implies that $\|\Psi_{m_1}^{-1}\|_{\rm op}\leqslant 1/\underline{\mathfrak m}$. So, for $t\in (0,T]$ (resp. $t\in [t_0,T]$ with a fixed $t_0 > 0$), the model selection criterion (\ref{model_selection_criterion}) (resp. (\ref{model_selection_criterion_bis})) may be simplified in another way:
\begin{eqnarray*}
 \widehat{\bf m}^* & = &
 \arg\min_{\mathbf m\in\mathcal M_{N}^{*}}
 \{-\|\widehat p_{\mathbf m,t}\|_{N}^{2} + 2\kappa{\rm pen}(\mathbf m)\}\\
 & &
 \hspace{3cm}
 \textrm{(resp. }\widetilde{\bf m}^* =
 \arg\min_{\mathbf m\in\mathcal M_{N}^{*}}
 \{-\|\widehat p_{\mathbf m,t}\|_{N}^{2} + 2\kappa_b{\rm pen}_b(\mathbf m)\}),
\end{eqnarray*}
where $\mathcal M_{N}^{*} :=\mathcal U_N\cap (\mathcal V_{N}^{*}\times\mathcal N)$ and
\begin{displaymath}
\mathcal V_{N}^{*} :=
\left\{m_1\in\mathcal N : m_1\leqslant
\frac{\mathfrak c_{\Lambda}}{2\mathfrak c_{\varphi}^{2}
(\underline{\mathfrak m}^{-1}\vee 1)}\cdot\frac{NT}{\log(NT)}\right\}.
\end{displaymath}
A result similar to Theorem \ref{risk_bound_adaptive_estimator} (resp. Proposition \ref{risk_bound_adaptive_estimator_bis}) may be established on the adaptive estimator $\widehat p_{\widehat{\bf m}^*,t}$ (resp. $\widehat p_{\widetilde{\bf m}^*,t}$) by taking $\widehat{\mathcal M}_N =\mathcal M_N =\mathfrak M_N :=\mathcal M_{N}^{*}$ in the proof (of Theorem \ref{risk_bound_adaptive_estimator} (resp. Proposition \ref{risk_bound_adaptive_estimator_bis})). This means that the collection in which $ \widehat{\bf m}^*$ or $ \widetilde{\bf m}^*$ are selected is larger and not random.
%

% Section : Numerical experiments.

%
\section{Numerical experiments}\label{section_numerical_experiments}
We propose a brief simulation study to illustrate our estimation method. The implementation is done using the Hermite basis defined by (\ref{Hermite_basis}) ($I = J =\mathbb R$). The Hermite polynomials are computed thanks to the recursion formula $H_{n + 1}(x) = 2xH_n(x) - 2nH_{n - 1}(x)$ with $H_0(x) = 1$ and $H_1(x) = x$ (see Abramowitz and Stegun \cite{AS64}, (22.7)).
\\
\\
We fix $t = 1$, so we can take the penalty involved in Proposition \ref{risk_bound_adaptive_estimator_bis}: ${\rm pen}_b(\mathbf m) = m_1\sqrt{m_2}/N$. Moreover, we choose $\kappa_b = 2$; a value obtained from preliminary calibration experiments. A cutoff test excludes the dimensions $m_1$ such that the largest eigenvalue of $\widehat\Psi_{m_1}^{-1}$ is too large (see Comte and Genon-Catalot \cite{CGC20}).
\\
\\
We simulate discrete samples in three models, obtained from an exact discretization of $d$-dimensional Ornstein-Uhlenbeck processes $U_1,\dots,U_N$:
\begin{equation}\label{O-U}
dU_i(t) =
-\frac{r}{2}U_i(t)dt +
\frac{\gamma}{2}dW_{i,d}(t),
\quad U_i(0)\sim\mathcal N_d\left(0,\frac{\gamma^2}{4r}I_d\right),
\end{equation}
where $W_{i,d}$ is a $d$-dimensional standard Brownian motion. An exact simulation is generated with step $\Delta > 0$ by computing
\begin{displaymath}
U_i((k + 1)\Delta) =
e^{-\frac{r\Delta}{2}}U_i(k\Delta) +
\varepsilon_i((k + 1)\Delta),\quad
\varepsilon_i(k\Delta)\sim_{\rm iid}
\mathcal N_d\left(0,\frac{\gamma^2(1 - e^{-r\Delta})}{4r}I_d\right).
\end{displaymath}
In all cases, we take $k\in\{0,\dots,n\}$ with $n = 1000$, $\Delta = 0.01$, and as already mentioned, we fix $t = 1$ in the function $(x,y)\mapsto p_t(x,y)$ to estimate.
\\
\\
{\bf Example 1.} $X_i(t) = U_i(t)$, where $U_i(t)$ is defined by (\ref{O-U}) with $d = 1$. The $X_i$'s are independent copies of the solution of Equation (\ref{main_equation}) with
\begin{displaymath}
b(x) = -\frac{rx}{2},\quad
\sigma(x) =\frac{\gamma}{2},\quad
r = 2\quad {\rm and}\quad\gamma = 2.
\end{displaymath}
Here, the transition density function is given by
\begin{displaymath}
p_{t}^{(1)}(x,y) =
\sqrt{\frac{2r}{\pi\gamma^2(1 - e^{-rt})}}
\exp\left(-\frac{2r}{\gamma^2(1 - e^{-rt})}(y - xe^{-\frac{rt}{2}})^2\right).
\end{displaymath}
\begin{figure}[h!]
\includegraphics[width=14cm,height=5cm]{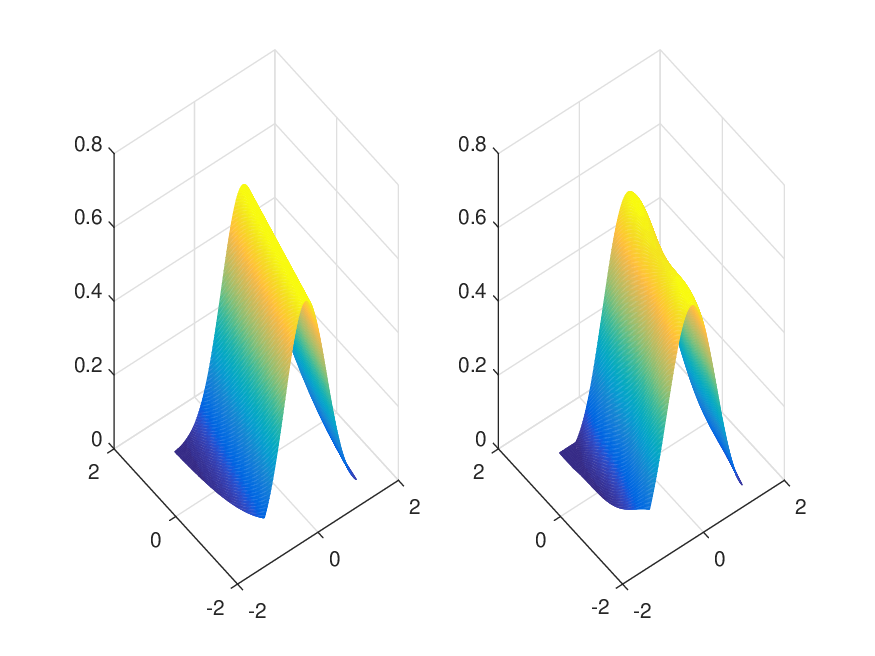}
\caption{Example 1. Transition density (left) and the estimation (right). Selected dimensions (4,5), 100*MISE = 0.22. $N = 200$, $T = 10$, $\Delta = 0.01$, $t = 1$.}
\label{Ex1D3}
\end{figure}
\begin{figure}[h!]
\includegraphics[width=12cm,height=8cm]{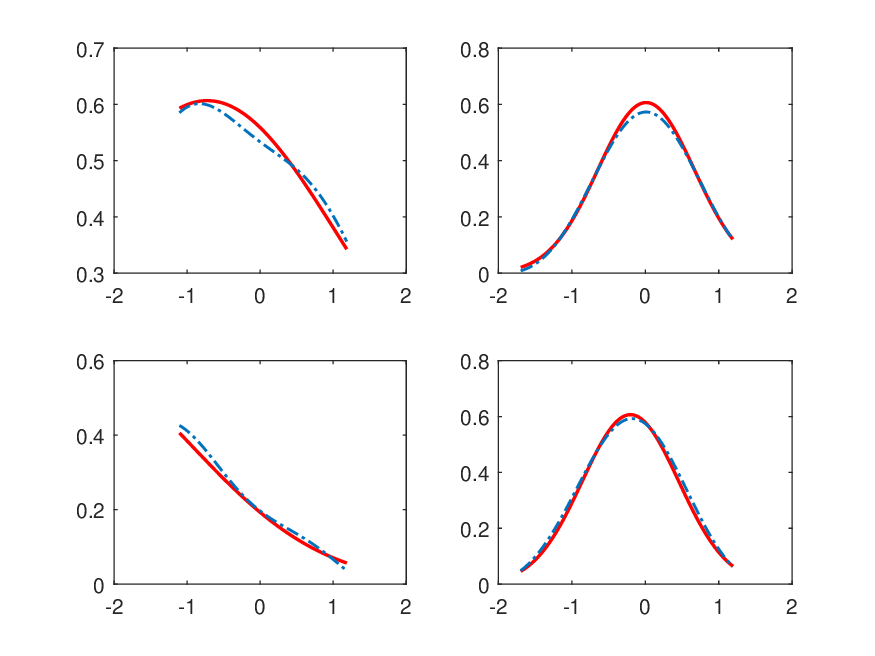}
\caption{Example 1. Full red line, the true and the estimation in dotted blue. Left: $x\mapsto p_t(x,y)$ for a fixed value of $y$ ($y = -0.27$ top and $y = -1$ bottom). Right: $y\mapsto p_t(x,y)$ for a fixed value of $x$ ($x = 0.03$ top and $x = -0.55$ bottom). $N = 200$, $T = 10$, $\Delta = 0.01$, $t = 1$.}
\label{Ex1Cut}
\end{figure}
\newline
{\bf Example 2.} $X_i(t) =\tanh(U_i(t))$, where $U_i(t)$ is defined by (\ref{O-U}) with $d = 1$. The $X_i$'s are independent copies of the solution of Equation (\ref{main_equation}) with
\begin{displaymath}
b(x) = (1 - x^2)\left(-\frac{r}{2}{\rm atanh}(x) -\frac{\gamma^2}{4}x\right),\quad
\sigma(x) =\frac{\gamma}{2}(1 - x^2),\quad r = 4\quad {\rm and}\quad\gamma = 1.
\end{displaymath}
Here, the transition density function is given by
\begin{displaymath}
p_{t}^{(2)}(x,y) =
\frac{p_{t}^{(1)}({\rm atanh}(x),{\rm atanh}(y))}{1 - y^2}.
\end{displaymath}
\begin{figure}[h!]
\includegraphics[width=14cm,height=5cm]{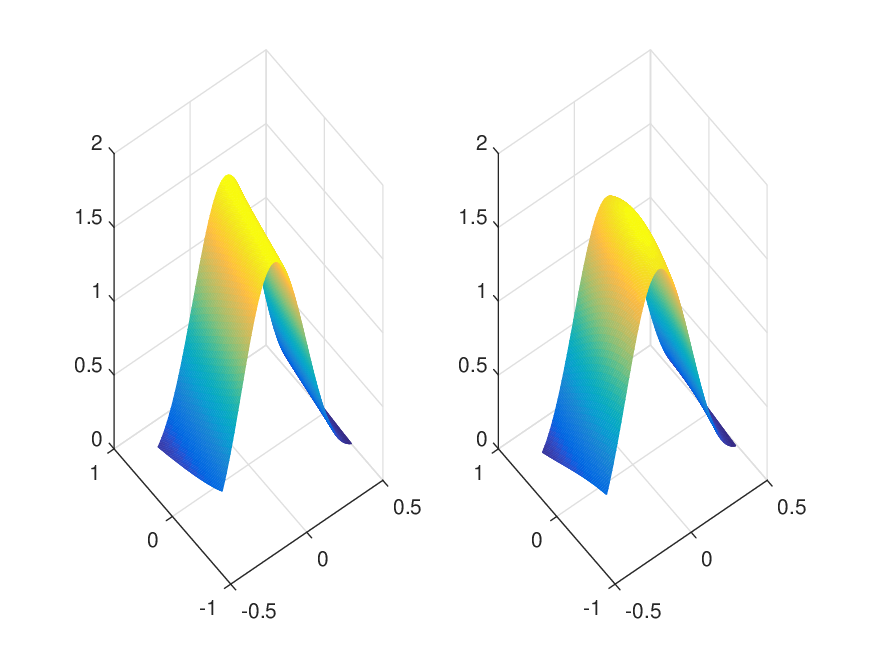}
\caption{Example 2. Transition density (left) and the estimation (right). Selected dimensions (2,41), 100*MISE = 0.16. $N = 200$, $T = 10$, $\Delta = 0.01$, $t = 1$.}
\label{Ex2D3}
\end{figure}
\newline
{\bf Example 3.} (Cox-Ingersoll-Ross or square-root process) $X_i(t) =\|U_i(t)\|_{2,d}^{2}$, where $U_i(t)$ is defined by (\ref{O-U}) with $d = 6$. The $X_i$'s are independent copies of the solution of Equation (\ref{main_equation}) with
\begin{displaymath}
b(x) =\frac{d\gamma^2}{4} - rx,\quad
\sigma(x) =\gamma\sqrt x,\quad
r = 1\quad {\rm and}\quad\gamma = 1.
\end{displaymath}
Here, the transition density function is given by
\begin{eqnarray*}
 p_{t}^{(3)}(x,y) & = &
 c_t\exp(-c_t(xe^{-rt} + y))\\
 & &
 \hspace{1cm}
 \times\left(\frac{y}{xe^{-rt}}\right)^{\frac{d}{4} -\frac{1}{2}}
 \mathcal I\left(\frac{d}{2} -1,2c_t\sqrt{xye^{-rt}}\right),
 \quad {\rm where}\quad
 c_t :=\frac{2r}{\gamma^2(1 - e^{-rt})}
\end{eqnarray*}
and $\mathcal I(p,x)$ is the modified Bessel function of the first kind of order $p$ at point $x$ (see the formula (20) in A\"it-Sahalia \cite{A99}).
\begin{figure}[h!]
\includegraphics[width=14cm,height=5cm]{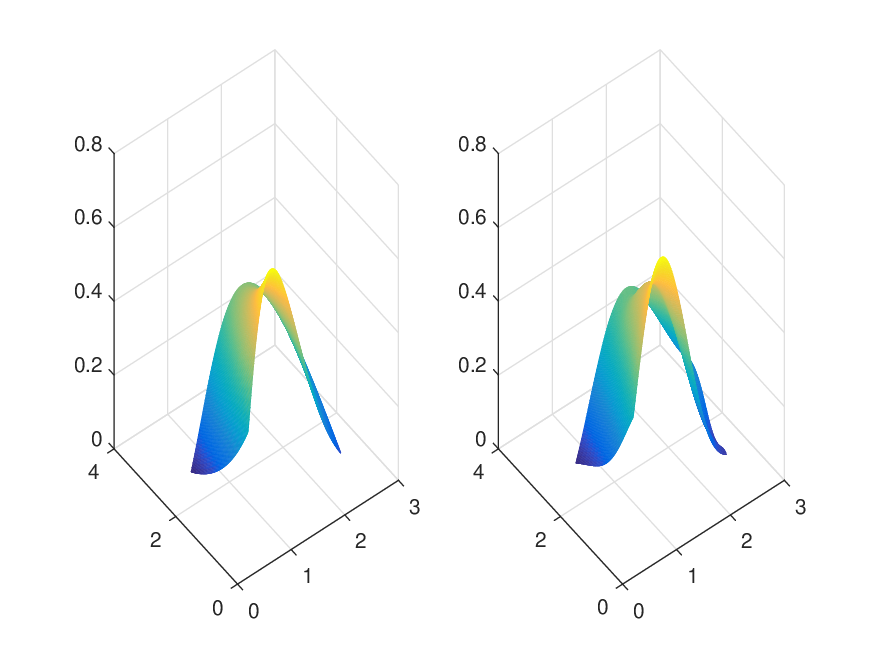}
\caption{Example 3. Transition density (left) and the estimation (right). Selected dimensions (6,9), 100*MISE = 0.29. $N = 200$, $T = 10$, $\Delta = 0.01$, $t = 1$.}
\label{Ex3D3}
\end{figure}
\begin{figure}[h!]
\includegraphics[width=14cm,height=5cm]{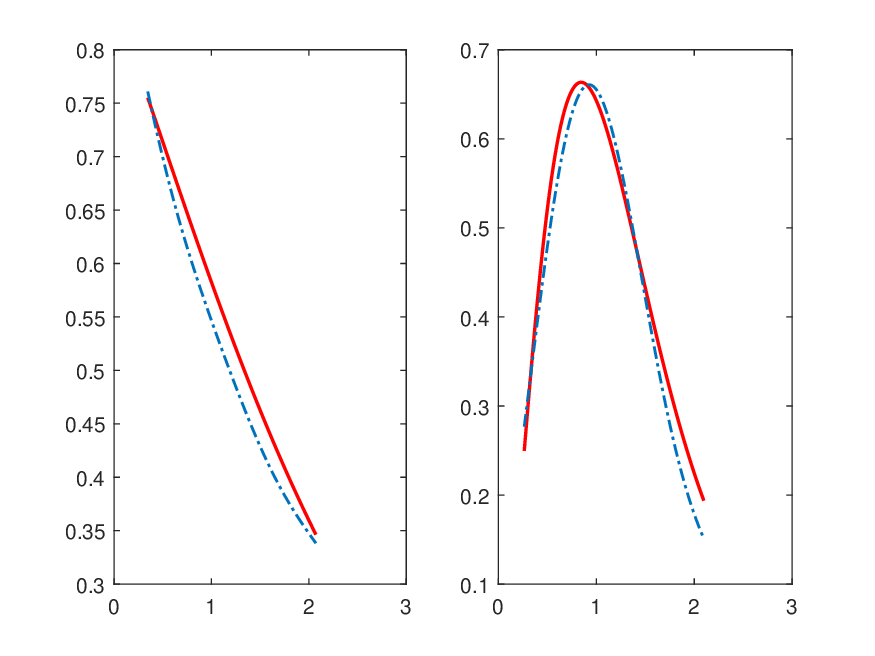}
\caption{Example 3. Full red line, the true and the estimation in dotted blue. Left: $x\mapsto p_t(x,y)$ for a fixed value of $y$ ($y = 0.71$). Right: $y\mapsto p_t(x,y)$ for a fixed value of $x$ ($x = 0.76$). $N = 200$, $T = 10$, $\Delta = 0.01$, $t = 1$.}
\label{Ex3Cut25}
\end{figure}
\newline
\newline
Figures \ref{Ex1D3}, \ref{Ex2D3} and \ref{Ex3D3} show the true surface and the estimated one for Examples 1, 2 and 3 respectively. We see that the {\it shape} of the estimated transition density is very similar to the true one, and the selected dimensions can be of any orders, especially in the $y$-direction (see Figure \ref{Ex2D3}, where the selected couple of dimensions is (2,41)). Figures \ref{Ex1Cut} and \ref{Ex3Cut25} represent sections of the curve for fixed values of $y$ or $x$ for the same simulated paths as in Figure \ref{Ex1D3} and \ref{Ex3D3} respectively. One may notice that the estimation in the $y$-direction is better than in the $x$-direction.
\begin{center}
\begin{table}
\begin{tabular}{cl|ccc}
 Model & & $N = 100$ & $N = 400$ & $N = 1000$\\
 \hline 
 & & & & \\
 1 & MISE & 1.62 (2.64) & 0.48 (0.99) & 0.18 (0.40)\\
 & Medians & 0.76 & 0.21 & 0.10\\
 & Dim (10,12) & (3.05,4.91) & (4.23,5.92) & (5.00,7.00)\\
 & & & & \\
 3 & MISE & 6.80 (10.5) & 2.14 (5.62) & 1.27 (5.62)\\
 & Medians & 1.65 & 0.34 & 0.19\\
 & dim (12,15) & (5.01, 6.68) & (7.07, 9.07) & (8.94, 12.2)\\
 & & & &
\end{tabular}
\caption{\footnotesize{Line MISE: 100*MISE (with 100*standard deviation) computed over 200 repetitions, by (\ref{MISECalcul}). Line Medians: Median values of 100*MISE. Line Dim with maximal proposals for $(m_1,m_2)$: means of the selected couples ($\widehat m_1,\widehat m_2)$.}}
\label{MISE}
\end{table}
\end{center}
In Table \ref{MISE}, we compute normalized squared errors for Models 1 and 3, which are defined by
\begin{equation}\label{MISECalcul}
{\rm MISE} =
\frac{\frac{1}{K}\sum_{k = 1}^{K}\frac{DXY^{(k)}}{N_IN_J}\sum_{i = 1}^{N_I}\sum_{j = 1}^{N_J}
(p_t(x_{i}^{(k)},y_{j}^{(k)}) -\widehat p_{t}^{(k)}(x_{i}^{(k)},y_{j}^{(k)}))^2}{\frac{DXY^{(K)}}{N_IN_J}
\sum_{i = 1}^{N_I}\sum_{j = 1}^{N_J}p_{t}^{2}(x_{i}^{(K)},y_{j}^{(K)})},
\end{equation}
where
\begin{displaymath}
DXY^{(k)} := (bX^{(k)} - aX^{(k)})(bY^{(k)} - aY^{(k)}),
\end{displaymath}
$bX^{(k)}$ and $aX^{(k)}$ are the 98\% and 2\% quantiles of the $X_t^{(k)}$'s, and $bY^{(k)}$ and $aY^{(k)}$ are the 99\% and 1\% quantiles of the $X_{t + 1}^{(k)}$'s. The super-index $k$ denotes the repetition number, the points $(x_i^{(k)},y_j^{(k)})$ are equispaced in the range of the observations of the path $k$, and we take $N_I = N_J = 100$ and $K=200$. We also give the associated standard deviation, together with the mean of the selected dimensions in each direction. To spare time of computation, we've adjusted maximal proposals for $m_1$ and $m_2$ to the choices corresponding to each example in such a way that the largest proposal is never selected (i.e. is always too large).\\
The results in Table \ref{MISE} show that, as could be expected, the error is getting smaller when $N$ increases, and in the same time, the selected dimensions are increasing. This is expected from Proposition \ref{empirical_risk_bound_Sobolev_Hermite}. We also notice that medians are much smaller in all cases than means: this indicates that the performance of the estimation is most of the time much better than what the mean indicates. Probably few bad results deteriorate the mean. The variability can be checked in the results to be unrelated to the selected dimensions, which are quite stable.
\\
\\
We have also implemented the (half)-trigonometric basis, or a product basis with $\varphi = t$ and $\psi = h$. They work well, with some questions on the nature of the theoretical bias term and underlying regularity spaces in the mixed ($t$ and $h$) case.
%

% Section : Concluding remarks.

%
\section{Concluding remarks}\label{section_conclusion}
In this paper, for a fixed $t\in (0,T]$, we have proposed a least squares contrast estimator $\widehat p_{\mathbf m,t}(x,y)$ of the transition density $p_t(x,y)$ of the solution $X$ to Equation (\ref{main_equation}). The estimator is defined through the estimators of the coefficients of its development on a basis of a finite dimensional space $\mathcal S_{\bf m} =\mathcal S_{\varphi,m_1}\otimes\mathcal S_{\psi,m_2}$, $\mathbf m = (m_1,m_2)$. In our observation setting, $N$ independent copies $X^1,\dots,X^N$ of $X$ are available. We provide an upper bound on a risk defined as the expectation of the empirical or integrated distance between $\widehat p_{{\mathbf m},t}$ and $p_t$, exhibiting a squared-bias/variance decomposition up to negligible terms. An adaptive procedure is then tailored to automatically select the couple $\mathbf m = (m_1,m_2)$, and it is proved to give good results, both in theory and in practice.\\
Of course, the simulation part immediately faces the question of handling high frequency but discrete samples and the topic may be developed. However, similarly to what happens in functional data analysis (FDA), small step samples often require a specific study from the beginning and can not be directly deduced from continuous-time constructions. In general, the continuous time results are extended to the high frequency one, up to some constraints linking the sample step and the number of observations. This would be worth being investigated. Note that FDA often considers both the influence of noise and discretization, which would mean here observations, for $i = 1,\dots,N$, of $X^i(t_j) +\eta_{i,j}$ with $t_j = j\Delta/n$, $j = 1,\dots,n$ and $\eta_{i,j}$ independent and identically distributed noises with common variance (see Chagny {\it et al.} \cite{CMR24}). It is not clear if both questions may be simultaneously solved in our setting.\\
Other questions may be: 1) To study functions related to the conditional density, the main example being the conditional cumulative distribution function; 2) A recent study by Amorino {\it et al.} \cite{AGH24} defines a concept of local differential privacy in the context of i.i.d. diffusion processes, and it may be worth studying the estimation of $f$ or $p_t$ under such confidentiality condition; (3) To extend the results of our paper - Theorems \ref{empirical_risk_bound}, \ref{f_weighted_risk_bound_truncated} and \ref{risk_bound_adaptive_estimator} - to an estimator minimizing a generalized objective function $\gamma_{d,N}$ for $d$-dimensional diffusion processes. As a matter of fact, if $X$ is a $d$-dimensional diffusion process with $d\in\mathbb N^*$, the objective function $\gamma_N$ is naturally extended by
\begin{equation}\label{multidimensional_objective_function}
\gamma_{d,N}(\tau) :=
\frac{1}{NT}\sum_{i = 1}^{N}\left(\int_{0}^{T}\int_{\mathbb R^d}\tau(X_{s}^{i},y)^2dyds -
2\int_{0}^{T}\tau(X_{s}^{i},X_{s + t}^{i})ds\right)
\textrm{ $;$ }
\tau\in\mathcal S_{\mathbf m_1,\dots,\mathbf m_d},
\end{equation}
where $\mathcal S_{\mathbf m_1,\dots,\mathbf m_d} :=\mathcal S_{\mathbf m_1}\otimes\dots\otimes\mathcal S_{\mathbf m_d}$ and $\mathbf m_1,\dots,\mathbf m_d\in\{1,\dots,N_T\}^2$. This is out of the scope of our paper, but note that one could get an expression of the solution of the minimization problem $\min_{\mathcal S_{\mathbf m_1,\dots,\mathbf m_d}}\gamma_{d,N}$ - similar to (\ref{definition_projection_LS}) - involving hyper-matrices in the spirit of Dussap \cite{DUSSAP21}. However, this requires the definition of additional tools, and if the theory can be generalized, the implementation becomes algorithmically difficult and a computer specialist matter.
\appendix
%

% Section : Proofs

%
\section{Proofs}\label{section_proofs}
%

% Subsection : Proof of Theorem empirical_risk_bound.

%
\subsection{Proof of Theorem \ref{empirical_risk_bound}}\label{section_proof_empirical_risk_bound}
The proof of Theorem \ref{empirical_risk_bound} relies on two technical lemmas stated first.
%

% Lemma : Probability of the "Omega event".

%
\begin{lemma}\label{Omega_event}
Consider the event
\begin{displaymath}
\Omega_{m_1} :=
\left\{\sup_{\tau\in\mathcal S_{\varphi,m_1}}
\left|\frac{\|\tau\|_{N,1}^{2}}{\|\tau\|_{f,1}^{2}} - 1\right|\leqslant\frac{1}{2}
\right\}
\end{displaymath}
where, for every $h\in\mathbb L^2(\mathbb R,f(x)dx)$,
\begin{displaymath}
\|h\|_{N,1}^{2} :=
\frac{1}{NT}\sum_{i = 1}^{N}\int_{0}^{T}h(X_{s}^{i})^2ds
\quad\textrm{and}\quad
\|h\|_{f,1}^{2} :=
\int_{-\infty}^{\infty}h(x)^2f(x)dx.
\end{displaymath}
Recall also that $\Lambda_{m_1} $ is defined by (\ref{event_Lambda}). Under Assumption \ref{condition_m}, there exists a constant $\mathfrak c_{\ref{Omega_event}} > 0$, not depending on $m_1$ and $N$, such that
\begin{displaymath}
\mathbb P(\Omega_{m_1}^{c})
\leqslant
\frac{\mathfrak c_{\ref{Omega_event}}}{N^p}
\quad\textrm{and}\quad
\mathbb P(\Lambda_{m_1}^{c})
\leqslant
\frac{\mathfrak c_{\ref{Omega_event}}}{N^p}
\quad {\rm with}\quad
p = 2(q + 4) + 1.
\end{displaymath}
\end{lemma}
\noindent
See Comte and Genon-Catalot \cite{CGC20}, Lemma 6.1 for a proof. Now, let us introduce two additional empirical maps:
\begin{itemize}
 \item The empirical process $\nu_N$, defined by
 \begin{displaymath}
 \nu_N(\tau) :=
 \frac{1}{NT}\sum_{i = 1}^{N}\int_{0}^{T}\left(
 \tau(X_{s}^{i},X_{s + t}^{i}) -\int_{-\infty}^{\infty}\tau(X_{s}^{i},y)p_t(X_{s}^{i},y)dy\right)ds
 \end{displaymath}
 for every $\tau\in\mathcal S_{\bf m}$. Note that for $\widehat Z_{\mathbf m,t}$ defined by (\ref{Z_vector}), it holds
 \begin{equation}\label{Z_empirical_process_relationship}
 [\widehat Z_{\mathbf m,t}]_{j,\ell} =
 \frac{1}{NT}\sum_{i = 1}^{N}
 \int_{0}^{T}\varphi_j(X_{s}^{i})\psi_{\ell}(X_{s + t}^{i})ds =
 \langle p_t,\varphi_j\otimes\psi_{\ell}\rangle_N +
 \nu_N(\varphi_j\otimes\psi_{\ell})
 \end{equation}
 for every $j\in\{1,\dots,m_1\}$ and $\ell\in\{1,\dots,m_2\}$.
 \item The empirical orthogonal projection $\widehat\Pi_{\bf m}$, defined by
 \begin{equation}\label{empirical_orthogonal_projection}
 \widehat\Pi_{\bf m}(\cdot)\in
 \arg\min_{\tau\in\mathcal S_{\bf m}}
 \|\tau -\cdot\|_{N}^{2}.
 \end{equation}
 For any function $h$ from $\mathbb R^2$ into $\mathbb R$,
 \begin{eqnarray*}
  \widehat\Pi_{\bf m}(h) & = &
  \sum_{j = 1}^{m_1}\sum_{\ell = 1}^{m_2}
  [\widehat\Psi_{m_1}^{-1}\widehat P_{\bf m}(h)]_{j,\ell}(\varphi_j\otimes\psi_{\ell})\\
  & &
  \hspace{1.5cm}{\rm with}\quad
  \widehat P_{\bf m}(h) = (\langle h,\varphi_j
  \otimes\psi_{\ell}\rangle_N)_{(j,\ell)\in\{1,\dots,m_1\}\times\{1,\dots,m_2\}}.
 \end{eqnarray*}
\end{itemize}
%

% Lemma : Suitable expression of the empirical distance between the estimator and the empirical projection of the transition density.

%
\begin{lemma}\label{empirical_risk_bound_key_lemma}
For every $t\in [0,T]$,
\begin{displaymath}
\|\widehat p_{\mathbf m,t} -\widehat\Pi_{\bf m}(p_t)\|_{N}^{2} =
\sup_{\tau\in\mathcal S_{\bf m} :\|\tau\|_N = 1}\nu_N(\tau)^2.
\end{displaymath}
\end{lemma}
\noindent
The proof of Lemma \ref{empirical_risk_bound_key_lemma} is postponed to Subsubsection \ref{section_proof_empirical_risk_bound_key_lemma}.
%

% Subsubsection : Steps of the proof.

%
\subsubsection{Steps of the proof}
First of all, by the definition of $\widehat\Pi_{\bf m}$ (see (\ref{empirical_orthogonal_projection})),
\begin{displaymath}
\|\widehat p_{\mathbf m,t} - p_{I\times J,t}\|_{N}^{2} =
\min_{\tau\in\mathcal S_{\bf m}}\|\tau - p_{I\times J,t}\|_{N}^{2} +
\|\widehat p_{\mathbf m,t} -\widehat\Pi_{\bf m}(p_t)\|_{N}^{2},
\end{displaymath}
and by Lemma \ref{empirical_risk_bound_key_lemma},
\begin{eqnarray*}
 \mathbb E(\|\widehat p_{\mathbf m,t} -\widehat\Pi_{\bf m}(p_t)\|_{N}^{2}) & = &
 \mathbb E\left(\mathbf 1_{\Omega_{m_1}}
 \sup_{\tau\in\mathcal S_{\bf m} :\|\tau\|_N = 1}\nu_N(\tau)^2\right) +
 \mathbb E\left(\mathbf 1_{\Omega_{m_1}^{c}}
 \sup_{\tau\in\mathcal S_{\bf m} :\|\tau\|_N = 1}\nu_N(\tau)^2\right)\\
 & =: & A + B.
\end{eqnarray*}
Let us find suitable bounds on $A$ and $B$.
\\
\\
{\bf Step 1 (bound on $A$).} First, since
\begin{displaymath}
\|.\|_{f,1}^{2}\mathbf 1_{\Omega_{m_1}}\leqslant
2\|.\|_{N,1}^{2}\mathbf 1_{\Omega_{m_1}}
\quad {\rm on}\quad
\mathcal S_{\varphi,m_1}
\end{displaymath}
by the definition of $\Omega_{m_1}$, for every $\tau\in\mathcal S_{\bf m}$,
\begin{eqnarray*}
 \|\tau\|_{f}^{2}\mathbf 1_{\Omega_{m_1}} & = &
 \int_{-\infty}^{\infty}\|\tau(\cdot,y)\|_{f,1}^{2}\mathbf 1_{\Omega_{m_1}}dy\\
 & &\hspace{2cm}
 \leqslant
 2\int_{-\infty}^{\infty}\|\tau(\cdot,y)\|_{N,1}^{2}\mathbf 1_{\Omega_{m_1}}dy =
 2\|\tau\|_{N}^{2}\mathbf 1_{\Omega_{m_1}}.
\end{eqnarray*}
Then,
\begin{displaymath}
\{\tau\in\mathcal S_{\bf m} :\|\tau\|_N = 1\}
\subset\{\tau\in\mathcal S_{\bf m} :\|\tau\|_{f}^{2}\leqslant 2\}
\quad {\rm on}\quad
\Omega_{m_1},
\end{displaymath}
leading to
\begin{displaymath}
A\leqslant
\mathbb E\left(
\sup_{\tau\in\mathcal S_{\bf m} :\|\tau\|_{f}^{2}\leqslant 2}\nu_N(\tau)^2\right).
\end{displaymath}
Since $(\varphi_1,\dots,\varphi_{m_1})$ is an orthonormal family of $\mathbb L^2(\mathbb R,dx)$, $\varphi_1,\dots,\varphi_{m_1}$ are linearly independent, and one may consider the basis $(\varphi_{1}^{f},\dots,\varphi_{m_1}^{f})$ of $\mathcal S_{\varphi,m_1}$, orthonormal in $\mathbb L^2(\mathbb R,f(x)dx)$, obtained from $(\varphi_1,\dots,\varphi_{m_1})$ via the Gram-Schmidt process. Thus, $(\varphi_{j}^{f}\otimes\psi_{\ell})_{j,\ell}$ is an orthonormal basis of $\mathcal S_{\bf m}$ equipped with $\langle .,.\rangle_f$, and denoting by $\|.\|_{2,\mathbf m}$ the Fr\"obenius norm on $\mathcal M_{m_1,m_2}(\mathbb R)$,
\begin{displaymath}
A\leqslant
\mathbb E\left[\sup_{\Theta :
\|\Theta\|_{2,\mathbf m}^{2}\leqslant 2}\left(
\sum_{j = 1}^{m_1}\sum_{\ell = 1}^{m_2}
\Theta_{j,\ell}\nu_N(\varphi_{j}^{f}\otimes\psi_{\ell})\right)^2\right]\\
\leqslant
2\mathbb E\left(\sum_{j = 1}^{m_1}\sum_{\ell = 1}^{m_2}
\nu_N(\varphi_{j}^{f}\otimes\psi_{\ell})^2\right).
\end{displaymath}
Now, note that for every $j\in\{1,\dots,m_1\}$ and $\ell\in\{1,\dots,m_2\}$,
\begin{eqnarray*}
 \mathbb E(\nu_N(\varphi_{j}^{f}\otimes\psi_{\ell}))
 & = &
 \frac{1}{NT}
 \mathbb E\left(
 \sum_{i = 1}^{N}\int_{0}^{T}\varphi_{j}^{f}(X_{s}^{i})
 (\psi_{\ell}(X_{s + t}^{i}) -\mathbb E(\psi_{\ell}(X_{s + t}^{i})|X_{s}^{i}))ds\right)\\
 & = &
 \frac{1}{T}
 \int_{0}^{T}\mathbb E(\varphi_{j}^{f}(X_s)\psi_{\ell}(X_{s + t}) -
 \mathbb E(\varphi_{j}^{f}(X_s)\psi_{\ell}(X_{s + t})|X_s))ds = 0.
\end{eqnarray*}
Therefore, since $\|\varphi_{j}^{f}\|_f = 1$ for every $j\in\{1,\dots,m_1\}$,
\begin{eqnarray*}
 A & \leqslant &
 2\sum_{j = 1}^{m_1}\sum_{\ell = 1}^{m_2}{\rm var}(\nu_N(\varphi_{j}^{f}\otimes\psi_{\ell}))\\
 & \leqslant &
 \frac{2}{N}\sum_{j = 1}^{m_1}\sum_{\ell = 1}^{m_2}
 \mathbb E\left[\left(\frac{1}{T}\int_{0}^{T}\varphi_{j}^{f}(X_s)
 (\psi_{\ell}(X_{s + t}) -\mathbb E(\psi_{\ell}(X_{s + t})|X_s))ds\right)^2\right]\\
 & \leqslant &
 \frac{2}{NT}\sum_{j = 1}^{m_1}\sum_{\ell = 1}^{m_2}
 \int_{0}^{T}(\mathbb E(\varphi_{j}^{f}(X_s)^2\psi_{\ell}(X_{s + t})^2) +
 \mathbb E(\varphi_{j}^{f}(X_s)^2\mathbb E(\psi_{\ell}(X_{s + t})|X_s)^2)\\
 & &
 \hspace{4cm} -
 2\mathbb E(\varphi_{j}^{f}(X_s)\psi_{\ell}(X_{s + t})
 \mathbb E(\varphi_{j}^{f}(X_s)\psi_{\ell}(X_{s + t})|X_s)))ds\\
 & \leqslant &
 \frac{2}{NT}\sum_{j = 1}^{m_1}\sum_{\ell = 1}^{m_2}
 \int_{0}^{T}\mathbb E(\varphi_{j}^{f}(X_s)^2\psi_{\ell}(X_{s + t})^2)ds
 \leqslant
 \frac{2m_1\mathfrak L_{\psi}(m_2)}{N}.
\end{eqnarray*}
{\bf Step 2 (bound on $B$).} Since $\varphi_1,\dots,\varphi_{m_1}$ are linearly independent as mentioned in Step 1, one may consider the basis $(\varphi_{1}^{N},\dots,\varphi_{m_1}^{N})$ of $\mathcal S_{\varphi,m_1}$, orthonormal for the empirical inner product $\langle .,.\rangle_{N,1}$, obtained from $(\varphi_1,\dots,\varphi_{m_1})$ via the Gram-Schmidt process. Then, $(\varphi_{j}^{N}\otimes\psi_{\ell})_{j,\ell}$ is an orthonormal basis of $\mathcal S_{\bf m}$ equipped with $\langle .,.\rangle_N$, and by Cauchy-Schwarz's and Jensen's inequalities,
\begin{eqnarray}
 \sup_{\tau\in\mathcal S_{\bf m} :\|\tau\|_N = 1}\nu_N(\tau)^2 & = &
 \sup_{\Theta :
 \|\Theta\|_{2,\mathbf m}^{2} = 1}\left(
 \sum_{j = 1}^{m_1}\sum_{\ell = 1}^{m_2}
 \Theta_{j,\ell}\nu_N(\varphi_{j}^{N}\otimes\psi_{\ell})\right)^2
 \nonumber\\
 & \leqslant &
 \sum_{j = 1}^{m_1}\sum_{\ell = 1}^{m_2}
 \left(\frac{1}{NT}\sum_{i = 1}^{N}\int_{0}^{T}\varphi_{j}^{N}(X_{s}^{i})
 (\psi_{\ell}(X_{s + t}^{i}) -\mathbb E(\psi_{\ell}(X_{s + t}^{i})|X_{s}^{i}))ds\right)^2
 \nonumber\\
 \label{empirical_risk_bound_2}
 & \leqslant &
 4\mathfrak L_{\psi}(m_2)
 \sum_{j = 1}^{m_1}\underbrace{
 \left(\frac{1}{NT}\sum_{i = 1}^{N}\int_{0}^{T}
 \varphi_{j}^{N}(X_{s}^{i})^2ds\right)}_{= 1} =
 4m_1\mathfrak L_{\psi}(m_2).
\end{eqnarray}
Therefore, by Assumption \ref{condition_m} and Lemma \ref{Omega_event}, there exists a constant $\mathfrak c_1 > 0$, not depending on $\mathbf m$, $N$ and $t$, such that
\begin{displaymath}
B\leqslant
4m_1\mathfrak L_{\psi}(m_2)\mathbb P(\Omega_{m_1}^{c})\leqslant
\mathfrak c_1N^{-1}.
\end{displaymath}
{\bf Step 3 (conclusion).} By the two previous steps,
\begin{displaymath}
\mathbb E(\|\widehat p_{\mathbf m,t} -\widehat\Pi_{\bf m}(p_t)\|_{N}^{2})
\leqslant\frac{2m_1\mathfrak L_{\psi}(m_2)}{N} +\frac{\mathfrak c_1}{N}.
\end{displaymath}
Thus,
\begin{eqnarray*}
 \mathbb E(\|\widehat p_{\mathbf m,t} - p_{I\times J,t}\|_{N}^{2}) & \leqslant &
 \min_{\tau\in\mathcal S_{\bf m}}\mathbb E(\|\tau - p_{I\times J,t}\|_{N}^{2}) +
 \mathbb E(\|\widehat p_{\mathbf m,t} -\widehat\Pi_{\bf m}(p_t)\|_{N}^{2})\\
 & \leqslant &
 \min_{\tau\in\mathcal S_{\bf m}}\|\tau - p_{I\times J,t}\|_{f}^{2} +
 \frac{2m_1\mathfrak L_{\psi}(m_2)}{N} +\frac{\mathfrak c_1}{N}.
\end{eqnarray*}
\hfill $\Box$
%

% Subsubsection : Proof of Lemma empirical_risk_bound_key_lemma.

%
\subsubsection{Proof of Lemma \ref{empirical_risk_bound_key_lemma}}\label{section_proof_empirical_risk_bound_key_lemma}
First, by Cauchy-Schwarz's inequality,
\begin{displaymath}
\sup_{\tau\in\mathcal S_{\bf m} :\|\tau\|_N = 1}
\langle\widehat p_{\mathbf m,t} -\widehat\Pi_{\bf m}(p_t),\tau\rangle_{N}^{2}
\leqslant
\|\widehat p_{\mathbf m,t} -\widehat\Pi_{\bf m}(p_t)\|_{N}^{2},
\end{displaymath}
and since
\begin{displaymath}
\tau^* :=
\frac{\widehat p_{\mathbf m,t} -
\widehat\Pi_{\bf m}(p_t)}{\|\widehat p_{\mathbf m,t} -
\widehat\Pi_{\bf m}(p_t)\|_N}\in\mathcal S_{\bf m}
\end{displaymath}
satisfies both $\|\tau^*\|_N = 1$ and
\begin{displaymath}
\langle\widehat p_{\mathbf m,t} -\widehat\Pi_{\bf m}(p_t),\tau^*\rangle_{N}^{2} =
\|\widehat p_{\mathbf m,t} -\widehat\Pi_{\bf m}(p_t)\|_{N}^{2},
\end{displaymath}
then
\begin{equation}\label{empirical_risk_bound_key_lemma_1}
\|\widehat p_{\mathbf m,t} -\widehat\Pi_{\bf m}(p_t)\|_{N}^{2} =
\sup_{\tau\in\mathcal S_{\bf m} :\|\tau\|_N = 1}
\langle\widehat p_{\mathbf m,t} -\widehat\Pi_{\bf m}(p_t),\tau\rangle_{N}^{2}.
\end{equation}
Now, let us show that
\begin{equation}\label{empirical_risk_bound_key_lemma_2}
\langle\widehat p_{\mathbf m,t} -
\widehat\Pi_{\bf m}(p_t),\tau\rangle_N =\nu_N(\tau)
\textrm{ $;$ }\forall\tau\in\mathcal S_{\bf m}.
\end{equation}
By the decompositions of $\widehat p_{\mathbf m,t}$ and $\widehat\Pi_{\bf m}(p_t)$ in the basis $(\varphi_j\otimes\psi_{\ell})_{j,\ell}$ of $\mathcal S_{\bf m}$, and by (\ref{Z_empirical_process_relationship}),
\begin{eqnarray*}
 \widehat p_{\mathbf m,t} -\widehat\Pi_{\bf m}(p_t) & = &
 \sum_{j = 1}^{m_1}\sum_{\ell = 1}^{m_2}
 [\widehat\Psi_{m_1}^{-1}(\widehat Z_{\mathbf m,t} -\widehat P_{\bf m}(p_t))]_{j,\ell}
 (\varphi_j\otimes\psi_{\ell})
  = 
 \sum_{j = 1}^{m_1}\sum_{\ell = 1}^{m_2}
 [\widehat\Psi_{m_1}^{-1}\widehat\Delta_{\bf m}]_{j,\ell}(\varphi_j\otimes\psi_{\ell}),
\end{eqnarray*}
where
\begin{eqnarray*}
 \widehat\Delta_{\bf m} & := &
 \left(\frac{1}{NT}\sum_{i = 1}^{N}\int_{0}^{T}
 \varphi_j(X_{s}^{i})\psi_{\ell}(X_{s + t}^{i})ds\right.\\
 & &
 \hspace{2cm}\left. -
 \frac{1}{NT}\sum_{i = 1}^{N}\int_{0}^{T}\int_{-\infty}^{\infty}
 \varphi_j(X_{s}^{i})\psi_{\ell}(y)p_t(X_{s}^{i},y)dyds\right)_{j,\ell} =
 (\nu_N(\varphi_j\otimes\psi_{\ell}))_{j,\ell}.
\end{eqnarray*}
Then, for any $\tau =\sum_{j,\ell}\Theta_{j,\ell}(\varphi_j\otimes\psi_{\ell})$ with $\Theta\in\mathcal M_{m_1,m_2}(\mathbb R)$,
\begin{eqnarray*}
 \langle\widehat p_{\mathbf m,t} -
 \widehat\Pi_{\bf m}(p_t),\tau\rangle_N & = &
 \sum_{j,j' = 1}^{m_1}\sum_{\ell,\ell' = 1}^{m_2}
 [\widehat\Psi_{m_1}^{-1}\widehat\Delta_{\bf m}]_{j,\ell}\Theta_{j',\ell'}
 \langle\varphi_j\otimes\psi_{\ell},\varphi_{j'}\otimes\psi_{\ell'}\rangle_N
 \end{eqnarray*} 
 As 
 $$\langle\varphi_j\otimes\psi_{\ell},\varphi_{j'}\otimes\psi_{\ell'}\rangle_N = \frac{1}{NT}\sum_{i = 1}^{N}
 \int_{0}^{T}\varphi_j(X_{s}^{i})\varphi_{j'}(X_{s}^{i})ds
 \int_{-\infty}^{\infty}\psi_{\ell}(y)\psi_{\ell'}(y)dy 
 = [\widehat\Psi_{m_1}]_{j,j'}\delta_{\ell,\ell'}$$
 we get 
 \begin{eqnarray*}
 \langle\widehat p_{\mathbf m,t} -
 \widehat\Pi_{\bf m}(p_t),\tau\rangle_N
 & = &
 \sum_{j,j' = 1}^{m_1}\sum_{\ell,\ell' = 1}^{m_2}
 [\widehat\Psi_{m_1}^{-1}\widehat\Delta_{\bf m}]_{j,\ell}\Theta_{j',\ell'}[\widehat\Psi_{m_1}]_{j,j'}\delta_{\ell,\ell'}\\
  \\
 & = &
 \sum_{j' = 1}^{m_1}\sum_{\ell = 1}^{m_2}\Theta_{j',\ell}
 \underbrace{\sum_{j = 1}^{m_1}[\widehat\Psi_{m_1}]_{j,j'}
 [\widehat\Psi_{m_1}^{-1}\widehat\Delta_{\bf m}]_{j,\ell}}_{= [\widehat\Delta_{\bf m}]_{j',\ell}}\\
 & = &
 \sum_{j' = 1}^{m_1}\sum_{\ell = 1}^{m_2}\Theta_{j',\ell}\nu_N(\varphi_{j'}\otimes\psi_{\ell}) =
 \nu_N(\tau).
\end{eqnarray*}
Therefore, by Equalities (\ref{empirical_risk_bound_key_lemma_1}) and (\ref{empirical_risk_bound_key_lemma_2}) together,
\begin{displaymath}
\|\widehat p_{\mathbf m,t} -\widehat\Pi_{\bf m}(p_t)\|_{N}^{2} =
\sup_{\tau\in\mathcal S_{\bf m} :\|\tau\|_N = 1}\nu_N(\tau)^2.
\end{displaymath}
\hfill $\Box$
%

% Subsection : Proof of Theorem f_weighted_risk_bound_truncated.

%
\subsection{Proof of Theorem \ref{f_weighted_risk_bound_truncated}}\label{section_proof_f_weighted_risk_bound_truncated}
First of all,
\begin{eqnarray*}
 \mathbb E(\|\widetilde p_{\mathbf m,t} - p_{I\times J,t}\|_{f}^{2}) & = &
 \mathbb E(\|\widetilde p_{\mathbf m,t} - p_{I\times J,t}\|_{f}^{2}\mathbf 1_{\Omega_{m_1}}) +
 \mathbb E(\|\widetilde p_{\mathbf m,t} - p_{I\times J,t}\|_{f}^{2}\mathbf 1_{\Omega_{m_1}^{c}})\\
 & =: &
 A + B.
\end{eqnarray*}
Let us find suitable bounds on $A$ and $B$.
\\
\\
{\bf Step 1 (bound on $A$).} As already mentioned, since
\begin{displaymath}
\|.\|_{f,1}^{2}\mathbf 1_{\Omega_{m_1}}\leqslant
2\|.\|_{N,1}^{2}\mathbf 1_{\Omega_{m_1}}
\quad {\rm on}\quad
\mathcal S_{\varphi,m_1}
\end{displaymath}
by the definition of $\Omega_{m_1}$, for every $\tau\in\mathcal S_{\bf m}$,
\begin{eqnarray*}
 \|\tau\|_{f}^{2}\mathbf 1_{\Omega_{m_1}} & = &
 \int_{-\infty}^{\infty}\|\tau(\cdot,y)\|_{f,1}^{2}\mathbf 1_{\Omega_{m_1}}dy\\
 & &\hspace{2cm}
 \leqslant
 2\int_{-\infty}^{\infty}\|\tau(\cdot,y)\|_{N,1}^{2}\mathbf 1_{\Omega_{m_1}}dy =
 2\|\tau\|_{N}^{2}\mathbf 1_{\Omega_{m_1}}.
\end{eqnarray*}
Then, noting $p_{I\times J,t}^{f}$ as the orthogonal projection of $p_{I\times J,t}$ onto $\mathcal S_{\bf m}$ for the inner product $\langle .,.\rangle_f$,
\begin{eqnarray*}
 \|\widetilde p_{\mathbf m,t} - p_{I\times J,t}\|_{f}^{2}\mathbf 1_{\Omega_{m_1}} & = &
 (\|\widetilde p_{\mathbf m,t} - p_{I\times J,t}^{f}\|_{f}^{2} +
 \|p_{I\times J,t}^{f} - p_{I\times J,t}\|_{f}^{2})\mathbf 1_{\Omega_{m_1}}\\
 & \leqslant &
 \|p_{I\times J,t}^{f} - p_{I\times J,t}\|_{f}^{2} +
 2\|\widetilde p_{\mathbf m,t} - p_{I\times J,t}^{f}\|_{N}^{2}\mathbf 1_{\Omega_{m_1}}\\
 & \leqslant &
 \min_{\tau\in\mathcal S_{\bf m}}
 \|\tau - p_{I\times J,t}\|_{f}^{2} +
 4\|\widetilde p_{\mathbf m,t} - p_{I\times J,t}\|_{N}^{2} +
 4\|p_{I\times J,t} - p_{I\times J,t}^{f}\|_{N}^{2}.
\end{eqnarray*}
Moreover,
\begin{displaymath}
\|\widetilde p_{\mathbf m,t} - p_{I\times J,t}\|_{N}^{2}\leqslant
\|\widehat p_{\mathbf m,t} - p_{I\times J,t}\|_{N}^{2} +
\|p_{I\times J,t}\|_{N}^{2}\mathbf 1_{\Lambda_{m_1}^{c}}
\end{displaymath}
and, by Lemma \ref{Omega_event},
\begin{eqnarray*}
 \mathbb E(\|p_{I\times J,t}\|_{N}^{2}\mathbf 1_{\Lambda_{m_1}^{c}})
 & \leqslant &
 \mathbb E\left[\left(\frac{1}{NT}\sum_{i = 1}^{N}
 \int_{0}^{T}\int_{-\infty}^{\infty}
 p_t(X_{s}^{i},y)^2dyds\right)^2\right]^{\frac{1}{2}}
 \mathbb P(\Lambda_{m_1}^{c})^{\frac{1}{2}}\\
 & \leqslant &
 \mathfrak c_{\ref{Omega_event}}^{\frac{1}{2}}\frac{R(t)}{N^{p/2}}
 \quad {\rm with}\quad
 R(t) =
 \frac{1}{T}\mathbb E\left[\left(\int_{0}^{T}\int_{-\infty}^{\infty}
 p_t(X_s,y)^2dyds\right)^2\right]^{\frac{1}{2}}.
\end{eqnarray*}
Therefore, by Theorem \ref{empirical_risk_bound}, there exists a constant $\mathfrak c_1 > 0$, not depending on $\mathbf m$, $N$ and $t$, such that
\begin{eqnarray*}
 A & \leqslant &
 5\min_{\tau\in\mathcal S_{\bf m}}
 \|\tau - p_{I\times J,t}\|_{f}^{2} +
 4\mathbb E(\|\widetilde p_{\mathbf m,t} - p_{I\times J,t}\|_{N}^{2})\\
 & \leqslant &
 9\min_{\tau\in\mathcal S_{\bf m}}
 \|\tau - p_{I\times J,t}\|_{f}^{2} +
 \frac{8m_1\mathfrak L_{\psi}(m_2)}{N} +
 \frac{\mathfrak c_1(1 + R(t))}{N}.
\end{eqnarray*}
{\bf Step 2 (bound on $B$).} Since $\widehat\Theta_{\mathbf m,t}\widehat\Theta_{\mathbf m,t}^{*}$ is a positive semidefinite (symmetric) matrix,
\begin{eqnarray*}
 \|\widehat p_{\mathbf m,t}\|_{f}^{2} & = &
 \int_{-\infty}^{\infty}f(x)\int_{-\infty}^{\infty}\left(
 \sum_{j = 1}^{m_1}\sum_{\ell = 1}^{m_2}
 [\widehat\Theta_{\mathbf m,t}]_{j,\ell}\varphi_j(x)\psi_{\ell}(y)\right)^2dydx\\
 & = &
 \int_{-\infty}^{\infty}f(x)
 \sum_{j,j' = 1}^{m_1}\sum_{\ell,\ell' = 1}^{m_2}
 [\widehat\Theta_{\mathbf m,t}]_{j,\ell}
 [\widehat\Theta_{\mathbf m,t}]_{j',\ell'}
 \varphi_j(x)\varphi_{j'}(x)
\int_{-\infty}^{\infty}
 \psi_{\ell}(y)\psi_{\ell'}(y)dydx
 \end{eqnarray*} 
 Now using that $\int_{-\infty}^{\infty}
 \psi_{\ell}(y)\psi_{\ell'}(y)dy=\delta_{\ell,\ell'}$, we get 
 \begin{eqnarray*}
 \|\widehat p_{\mathbf m,t}\|_{f}^{2} 
 & = &
 \sum_{j,j' = 1}^{m_1}\sum_{\ell = 1}^{m_2}
 [\widehat\Theta_{\mathbf m,t}]_{j,\ell}
 [\widehat\Theta_{\mathbf m,t}]_{j',\ell}
\int_{-\infty}^{\infty}\varphi_j(x)\varphi_{j'}(x)f(x)dx \\
 & = &
 \sum_{j = 1}^{m_1}\sum_{\ell = 1}^{m_2}
 [\widehat\Theta_{\mathbf m,t}^{*}]_{\ell,j}
 \underbrace{\sum_{j' = 1}^{m_1}
 [\Psi_{m_1}]_{j,j'}
 [\widehat\Theta_{\mathbf m,t}]_{j',\ell}}_{= [\Psi_{m_1}\widehat\Theta_{\mathbf m,t}]_{j,\ell}} \mbox{ as } \int_{-\infty}^{\infty}\varphi_j(x)\varphi_{j'}(x)f(x)dx = [\Psi_{m_1}]_{j,j'} \\
 & = &
 \sum_{j = 1}^{m_1}
 [\Psi_{m_1}\widehat\Theta_{\mathbf m,t}\widehat\Theta_{\mathbf m,t}^{*}]_{j,j} =
 {\rm trace}(\Psi_{m_1}\widehat\Theta_{\mathbf m,t}\widehat\Theta_{\mathbf m,t}^{*})
 \leqslant\|\Psi_{m_1}\|_{\rm op}
 {\rm trace}(\widehat\Theta_{\mathbf m,t}\widehat\Theta_{\mathbf m,t}^{*}).
\end{eqnarray*}
Let us find suitable controls on $\|\Psi_{m_1}\|_{\rm op}^{2}$ and $\mathbb E({\rm trace}(\widehat\Theta_{\mathbf m,t}\widehat\Theta_{\mathbf m,t}^{*})^2)$. On the one hand, by Cauchy-Schwarz's and Jensen's inequalities,
\begin{eqnarray*}
 \|\Psi_{m_1}\|_{\rm op}^{2} & = &
 \sup_{y :\|y\|_{2,m_1} = 1}
 \sum_{j = 1}^{m_1}\left(\sum_{j' = 1}^{m_1}
 y_{j'}\int_{-\infty}^{\infty}\varphi_j(x)\varphi_{j'}(x)f(x)dx\right)^2\\
 & \leqslant &
 \sum_{j,j' = 1}^{m_1}
 \left(\int_{-\infty}^{\infty}\varphi_j(x)\varphi_{j'}(x)f(x)dx\right)^2\\
 & \leqslant &
 \sum_{j,j' = 1}^{m_1}
 \int_{-\infty}^{\infty}\varphi_j(x)^2\varphi_{j'}(x)^2f(x)dx
 \leqslant
 \mathfrak L_{\varphi}(m_1)^2\underbrace{\int_{-\infty}^{\infty}f(x)dx}_{= 1}.
\end{eqnarray*}
On the other hand, since $\widehat Z_{\mathbf m,t}\widehat Z_{\mathbf m,t}^{*}$ is a positive semidefinite (symmetric) matrix,
\begin{displaymath}
{\rm trace}(\widehat\Theta_{\mathbf m,t}\widehat\Theta_{\mathbf m,t}^{*}) =
{\rm trace}((\widehat\Psi_{m_1}^{-1})^2\widehat Z_{\mathbf m,t}
\widehat Z_{\mathbf m,t}^{*})
\leqslant
\|\widehat\Psi_{m_1}^{-1}\|_{\rm op}^{2}
{\rm trace}(\widehat Z_{\mathbf m,t}\widehat Z_{\mathbf m,t}^{*}).
\end{displaymath}
Moreover,
\begin{eqnarray*}
 \mathbb E({\rm trace}(\widehat Z_{\mathbf m,t}\widehat Z_{\mathbf m,t}^{*})^2)
 & = &
 \mathbb E\left[\left(\sum_{j = 1}^{m_1}\sum_{\ell = 1}^{m_2}
 [\widehat Z_{\mathbf m,t}]_{j,\ell}^{2}\right)^2\right]\\
 & \leqslant &
 \frac{m_1m_2}{N^4T^4}\sum_{j = 1}^{m_1}\sum_{\ell = 1}^{m_2}
 \mathbb E\left[\left(\sum_{i = 1}^{N}
 \int_{0}^{T}\varphi_j(X_{s}^{i})\psi_{\ell}(X_{s + t}^{i})ds\right)^4\right]
 \leqslant
 m_1m_2\mathfrak L_{\varphi}(m_1)^2\mathfrak L_{\psi}(m_2)^2
\end{eqnarray*}
and, on the event $\Lambda_{m_1}$,
\begin{displaymath}
\|\widehat\Psi_{m_1}^{-1}\|_{\rm op}^{4}
\leqslant\mathfrak c_{\Lambda}^{4}\frac{N^4T^4}{\log(NT)^4\mathfrak L_{\varphi}(m_1)^4}.
\end{displaymath}
Then, there exists a constant $\mathfrak c_2 > 0$, not depending on $\mathbf m$, $N$ and $t$, such that
\begin{displaymath}
\mathbb E(\|\widetilde p_{\mathbf m,t}\|_{f}^{4})\leqslant
\mathfrak c_2N^4m_1m_2\mathfrak L_{\psi}(m_2)^2.
\end{displaymath}
So, by Lemma \ref{Omega_event}, there exists a constant $\mathfrak c_3 > 0$, not depending on $\mathbf m$, $N$ and $t$, such that
\begin{eqnarray*}
 B & \leqslant &
 \mathbb E(\|\widetilde p_{\mathbf m,t} - p_{I\times J,t}\|_{f}^{4})^{\frac{1}{2}}
 \mathbb P(\Omega_{m_1}^{c})^{\frac{1}{2}}\\
 & \leqslant &
 \mathfrak c_3(1 +\|p_{I\times J,t}\|_{f}^{2})N^{3 + q -\frac{p}{2}}
 \leqslant
 \frac{\mathfrak c_3(1 +\|p_{I\times J,t}\|_{f}^{2})}{N}
 \quad
 \textrm{because $p = 2(q + 4) + 1$ in Lemma \ref{Omega_event}.}
\end{eqnarray*}
{\bf Step 3 (conclusion).} By the two previous steps, there exists a constant $\mathfrak c_4 > 0$, not depending on $\mathbf m$, $N$ and $t$, such that
\begin{displaymath}
\mathbb E(\|\widetilde p_{\mathbf m,t} - p_{I\times J,t}\|_{f}^{2})
\leqslant
9\min_{\tau\in\mathcal S_{\bf m}}
\|\tau - p_{I\times J,t}\|_{f}^{2} +
\frac{8m_1\mathfrak L_{\psi}(m_2)}{N} +
\frac{\mathfrak c_4(1 + R_f(t))}{N}
\end{displaymath}
with $R_f(t) = R(t) +\|p_{I\times J,t}\|_{f}^{2}$. Moreover, by Inequality (\ref{Gaussian_bound_density}), and since $y\mapsto p_t(x,y)$ is a density function for every $x\in\mathbb R$,
\begin{eqnarray*}
 R(t)
 & = &
 \frac{1}{T}\mathbb E\left[\left(\int_{0}^{T}\int_{-\infty}^{\infty}
 p_t(X_s,y)^2dyds\right)^2\right]^{\frac{1}{2}}\\
 & &
 \hspace{2cm}\leqslant
 \frac{1}{T}\mathbb E\left[\left(
 \sup_{(u,v)\in\mathbb R^2}p_t(u,v)
 \int_{0}^{T}\int_{-\infty}^{\infty}
 p_t(X_s,y)dyds\right)^2\right]^{\frac{1}{2}}\\
 & &
 \hspace{2cm}\leqslant
 \frac{\overline{\mathfrak c}_T}{Tt^{1/2}}\mathbb E\left[\left(\int_{0}^{T}\int_{-\infty}^{\infty}
 p_t(X_s,y)dyds\right)^2\right]^{\frac{1}{2}} =
 \overline{\mathfrak c}_Tt^{-\frac{1}{2}}
\end{eqnarray*}
and
\begin{eqnarray*}
 \|p_{I\times J,t}\|_{f}^{2}
 & = &
 \int_I\int_J
 p_t(x,y)^2f(x)dxdy\\
 & &
 \hspace{2cm}\leqslant
 \sup_{(u,v)\in\mathbb R^2}p_t(u,v)
 \int_{-\infty}^{\infty}f(x)\int_{-\infty}^{\infty}
 p_t(x,y)dydx\\
 & &
 \hspace{2cm}\leqslant
 \overline{\mathfrak c}_Tt^{-\frac{1}{2}}
 \int_{-\infty}^{\infty}f(x)\int_{-\infty}^{\infty}
 p_t(x,y)dydx =\overline{\mathfrak c}_Tt^{-\frac{1}{2}},
\end{eqnarray*}
leading to
\begin{displaymath}
\frac{1}{T}\int_{0}^{T}R_f(t)dt\leqslant
4\overline{\mathfrak c}_TT^{-\frac{1}{2}}.
\end{displaymath}
In conclusion,
\begin{displaymath}
\frac{1}{T}\int_{0}^{T}
\mathbb E(\|\widetilde p_{\mathbf m,t} - p_{I\times J,t}\|_{f}^{2})dt
\leqslant
9\min_{\tau\in\mathcal S_{\bf m}}\left\{\frac{1}{T}\int_{0}^{T}
\|\tau - p_{I\times J,t}\|_{f}^{2}dt\right\} +
\frac{8m_1\mathfrak L_{\psi}(m_2)}{N} +
\frac{\mathfrak c_4(1 + 4\overline{\mathfrak c}_TT^{-1/2})}{N}.
\end{displaymath}
\hfill $\Box$
%

% Subsection : Proof of Theorem risk_bound_adaptive_estimator.

%
\subsection{Proof of Theorem \ref{risk_bound_adaptive_estimator}}\label{section_proof_risk_bound_adaptive_estimator}
The proof of Theorem \ref{risk_bound_adaptive_estimator} relies on the two following technical lemmas.
%

% Lemma : Projection of the estimator on a smaller model.

%
\begin{lemma}\label{projection_estimator_smaller_model}
Under Assumption \ref{assumption_projection_spaces}, for any $\mathbf m = (m_1,m_2)$ and $\mathbf M = (M_1,M_2)$ belonging to $\mathcal N^2$, if $\mathcal S_{\bf m}\subset\mathcal S_{\mathbf M}$, then
\begin{displaymath}
\widehat\Pi_{\bf m}(\widehat p_{\mathbf M,t}) =
\widehat p_{\mathbf m,t}.
\end{displaymath}
\end{lemma}
%

% Lemma : Probability of the "Xi event".

%
\begin{lemma}\label{Xi_event}
Consider the event
\begin{displaymath}
\Xi_N :=
\{\mathcal M_N\subset\widehat{\mathcal M}_N\subset\mathfrak M_N\},
\end{displaymath}
where $\mathfrak M_N :=\mathcal U_N\cap (\mathfrak V_N\times\mathcal N)$ and
\begin{displaymath}
\mathfrak V_N :=
\left\{m_1\in\mathcal N :
\mathfrak c_{\varphi}^{2}m_1(\|\Psi_{m_1}^{-1}\|_{\rm op}^{2}\vee 1)
\leqslant 4\mathfrak d\frac{NT}{\log(NT)}\right\}.
\end{displaymath}
Under Assumption \ref{assumption_projection_spaces}, there exists a constant $\mathfrak c_{\ref{Xi_event}} > 0$, not depending on $N$, such that
\begin{displaymath}
\mathbb P(\Xi_{N}^{c})\leqslant
\frac{\mathfrak c_{\ref{Xi_event}}}{N^{p - 1}}.
\end{displaymath}
\end{lemma}
\noindent
The proof of Lemma \ref{projection_estimator_smaller_model} is postponed to Subsubsection \ref{section_proof_projection_estimator_smaller_model}, and Lemma \ref{Xi_event} is a straightforward consequence of Comte and Genon-Catalot \cite{CGC20}, Inequality (6.17), because
\begin{displaymath}
\{\mathcal V_N\subset\widehat{\mathcal V}_N\subset\mathfrak V_N\}
\subset\Xi_N.
\end{displaymath}
Note that $\widehat{\bf m}$ needs to be selected in the empirical set $\widehat{\mathcal M}_N$, instead of $\mathcal M_N$, because $\mathcal M_N$ depends on $m_1\mapsto\|\Psi_{m_1}^{-1}\|_{\rm op}$ which has no explicit control in general. However, and this is the purpose of Lemma \ref{Xi_event}, we need that with high probability:
\begin{itemize}
 \item All the possible couples of dimensions in $\mathcal M_N$ may be selected in $\widehat{\mathcal M}_N$.
 \item The couple of dimensions $\widehat{\bf m}(\omega)$ selected in $\widehat{\mathcal M}_N(\omega)$ belongs to $\mathfrak M_N$ in order that $\Psi_{\widehat m_1(\omega)}$ satisfies Assumption \ref{condition_m}, and then to apply Lemma \ref{Omega_event} and Theorem \ref{empirical_risk_bound} in the (following) proof of Theorem \ref{risk_bound_adaptive_estimator}.
\end{itemize}
%

% Subsubsection : Steps of the proof.

%
\subsubsection{Steps of the proof}\label{section_steps_proof_theorem_risk_bound_adaptive_estimator}
The proof of Theorem \ref{risk_bound_adaptive_estimator} is dissected in four steps.
\\
\\
{\bf Step 1.} Let us prove that for any $\mathbf m = (m_1,m_2)\in\widehat{\mathcal M}_N$,
\begin{equation}\label{risk_bound_adaptive_estimator_1}
\|\widehat p_{\widehat{\bf m},t} - p_{I\times J,t}\|_{N}^{2}\leqslant
6\|\widehat p_{\mathbf m,t} - p_{I\times J,t}\|_{N}^{2} +
4\kappa{\rm pen}(\mathbf m) + 11\left(
\|\widehat p_{\widehat{\bf m},t} -\widehat{\Pi}_{\widehat{\bf m}}(p_t)\|_{N}^{2} -
\frac{2}{11}\kappa{\rm pen}(\widehat{\bf m})\right)_+.
\end{equation}
Consider $\mathbf M = (\max(\mathcal N),\max(\mathcal N))$ and $m\in\{\mathbf m,\widehat{\bf m}\}\subset\widehat{\mathcal M}_N\subset\mathcal N^2$. First, let us show that
\begin{eqnarray}
 \|\widehat p_{\widehat{\bf m},t} - p_{I\times J,t}\|_{N}^{2}
 \label{risk_bound_adaptive_estimator_2}
 & \leqslant &
 \|\widehat p_{\mathbf m,t} - p_{I\times J,t}\|_{N}^{2} -
 2(\|\widehat p_{\mathbf m,t} -\widehat\Pi_{\bf m}(p_t)\|_{N}^{2} -
 \kappa{\rm pen}(\mathbf m))\\
 & &
 \hspace{4cm} +
 2(\|\widehat p_{\widehat{\bf m},t} -\widehat\Pi_{\widehat{\bf m}}(p_t)\|_{N}^{2} -
 \kappa{\rm pen}(\widehat{\bf m})) + R_{\bf M}
 \nonumber
\end{eqnarray}
with
\begin{displaymath}
R_{\bf M} = 2\langle\widehat\Pi_{\widehat{\bf m}}(p_t) -
\widehat\Pi_{\bf m}(p_t),\widehat p_{\mathbf M,t} -\widehat\Pi_{\bf M}(p_t)\rangle_N.
\end{displaymath}
Since $\mathcal S_m\subset\mathcal S_{\bf M}$ by Assumption \ref{assumption_projection_spaces}.(1,2),
\begin{displaymath}
\widehat\Pi_m(\widehat p_{\mathbf M,t}) =\widehat p_{m,t}
\quad\textrm{by Lemma \ref{projection_estimator_smaller_model}.}
\end{displaymath}
Then,
\begin{eqnarray*}
 \|\widehat p_{\mathbf M,t} -\widehat p_{m,t}\|_{N}^{2} & = &
 \|\widehat p_{\mathbf M,t}\|_{N}^{2} +\|\widehat p_{m,t}\|_{N}^{2} -
 2\langle\widehat p_{\mathbf M,t},\widehat p_{m,t}\rangle_N\\
 & = &
 \|\widehat p_{\mathbf M,t}\|_{N}^{2} +\|\widehat\Pi_m(\widehat p_{\mathbf M,t})\|_{N}^{2}\\
 & &
 \hspace{2cm} -
 2\underbrace{\langle\widehat p_{\mathbf M,t} -\widehat\Pi_m(\widehat p_{\mathbf M,t}),
 \widehat\Pi_m(\widehat p_{\mathbf M,t})\rangle_N}_{= 0} -
 2\langle\widehat\Pi_m(\widehat p_{\mathbf M,t}),
 \widehat\Pi_m(\widehat p_{\mathbf M,t})\rangle_N\\
 & = &
 \|\widehat p_{\mathbf M,t}\|_{N}^{2} -\|\widehat p_{m,t}\|_{N}^{2}.
\end{eqnarray*}
So, by the definition of $\widehat{\bf m}$,
\begin{eqnarray*}
 \|\widehat p_{\mathbf M,t} -\widehat p_{\widehat{\bf m},t}\|_{N}^{2} +
 2\kappa{\rm pen}(\widehat{\bf m}) & = &
 \|\widehat p_{\mathbf M,t}\|_{N}^{2} -\|\widehat p_{\widehat{\bf m},t}\|_{N}^{2} +
 2\kappa{\rm pen}(\widehat{\bf m})\\
 & \leqslant &
 \|\widehat p_{\mathbf M,t}\|_{N}^{2} -\|\widehat p_{\mathbf m,t}\|_{N}^{2} +
 2\kappa{\rm pen}(\mathbf m) =
 \|\widehat p_{\mathbf M,t} -\widehat p_{\mathbf m,t}\|_{N}^{2} +
 2\kappa{\rm pen}(\mathbf m),
\end{eqnarray*}
and thus
\begin{eqnarray*}
 & &
 \|\widehat p_{\widehat{\bf m},t} - p_{I\times J,t}\|_{N}^{2} -
 \|\widehat p_{\mathbf m,t} - p_{I\times J,t}\|_{N}^{2}\\
 & &
 \hspace{3.5cm} =
 \|\widehat p_{\widehat{\bf m},t} -\widehat p_{\mathbf M,t}\|_{N}^{2} -
 \|\widehat p_{\mathbf m,t} -\widehat p_{\mathbf M,t}\|_{N}^{2} +
 2\langle\widehat p_{\widehat{\bf m},t} -
 \widehat p_{\mathbf m,t},\widehat p_{\mathbf M,t} - p_{I\times J,t}\rangle_N\\
 & &
 \hspace{3.5cm}\leqslant
 2\kappa{\rm pen}(\mathbf m) - 2\kappa{\rm pen}(\widehat{\bf m}) +
 2\langle\widehat p_{\widehat{\bf m},t} -
 \widehat p_{\mathbf m,t},\widehat p_{\mathbf M,t} - p_{I\times J,t}\rangle_N.
\end{eqnarray*}
Moreover, since $\widehat\Pi_m\circ\widehat\Pi_{\bf M} =\widehat\Pi_m$ and $\widehat\Pi_m(\widehat p_{\mathbf M,t}) =\widehat p_{m,t}$,
\begin{eqnarray*}
 & &
 \langle\widehat p_{m,t} -\widehat\Pi_m(p_t),
 \widehat p_{\mathbf M,t} -\widehat\Pi_{\bf M}(p_t)\rangle_N\\
 & &
 \hspace{2cm} =
 \underbrace{\langle\widehat\Pi_m(\widehat p_{\mathbf M,t} -\widehat\Pi_{\bf M}(p_t)),
 \widehat p_{\mathbf M,t} -\widehat\Pi_{\bf M}(p_t) -
 \widehat\Pi_m(\widehat p_{\mathbf M,t} -\widehat\Pi_{\bf M}(p_t))\rangle_N}_{= 0} +
 \|\widehat p_{m,t} -\widehat\Pi_m(p_t)\|_{N}^{2},
\end{eqnarray*}
leading to
\begin{eqnarray*}
 \langle\widehat p_{\widehat{\bf m},t} -
 \widehat p_{\mathbf m,t},\widehat p_{\mathbf M,t} - p_{I\times J,t}\rangle_N
 & = &
 \langle\widehat p_{\widehat{\bf m},t} -
 \widehat p_{\mathbf m,t},\widehat p_{\mathbf M,t} -\widehat\Pi_{\bf M}(p_t)\rangle_N +
 \underbrace{\langle\widehat p_{\widehat{\bf m},t} -
 \widehat p_{\mathbf m,t},\widehat\Pi_{\bf M}(p_t) - p_{I\times J,t}\rangle_N}_{= 0}\\
 & = &
 \langle\widehat p_{\widehat{\bf m},t} -
 \widehat\Pi_{\widehat{\bf m}}(p_t),\widehat p_{\mathbf M,t} -\widehat\Pi_{\bf M}(p_t)\rangle_N\\
 & &
 \hspace{2cm} +
 \langle\widehat\Pi_{\widehat{\bf m}}(p_t) -
 \widehat\Pi_{\bf m}(p_t),\widehat p_{\mathbf M,t} -\widehat\Pi_{\bf M}(p_t)\rangle_N\\
 & &
 \hspace{4cm} +
 \langle\widehat\Pi_{\bf m}(p_t) -
 \widehat p_{\mathbf m,t},\widehat p_{\mathbf M,t} -\widehat\Pi_{\bf M}(p_t)\rangle_N\\
 & = &
 \|\widehat p_{\widehat{\bf m},t} -\widehat\Pi_{\widehat{\bf m}}(p_t)\|_{N}^{2} -
 \|\widehat p_{\mathbf m,t} -\widehat\Pi_{\bf m}(p_t)\|_{N}^{2}\\
 & &
 \hspace{2cm} +
 \langle\widehat\Pi_{\widehat{\bf m}}(p_t) -
 \widehat\Pi_{\bf m}(p_t),\widehat p_{\mathbf M,t} -\widehat\Pi_{\bf M}(p_t)\rangle_N.
\end{eqnarray*}
Therefore,
\begin{eqnarray*}
 \|\widehat p_{\widehat{\bf m},t} - p_{I\times J,t}\|_{N}^{2}
 & \leqslant &
 \|\widehat p_{\mathbf m,t} - p_{I\times J,t}\|_{N}^{2} +
 2\kappa{\rm pen}(\mathbf m) - 2\kappa{\rm pen}(\widehat{\bf m}) +
 2\langle\widehat p_{\widehat{\bf m},t} -
 \widehat p_{\mathbf m,t},\widehat p_{\mathbf M,t} - p_{I\times J,t}\rangle_N\\
 & \leqslant &
 \|\widehat p_{\mathbf m,t} - p_{I\times J,t}\|_{N}^{2} -
 2(\|\widehat p_{\mathbf m,t} -\widehat\Pi_{\bf m}(p_t)\|_{N}^{2} -
 \kappa{\rm pen}(\mathbf m))\\
 & &
 \hspace{4cm} +
 2(\|\widehat p_{\widehat{\bf m},t} -\widehat\Pi_{\widehat{\bf m}}(p_t)\|_{N}^{2} -
 \kappa{\rm pen}(\widehat{\bf m})) + R_{\bf M}.
\end{eqnarray*}
Now, let us find a suitable bound on $R_{\bf M}$. Since $\varphi_1,\dots,\varphi_{\max(\mathcal N)}$ are linearly independent, one may consider the basis $(\varphi_{1}^{N},\dots,\varphi_{\max(\mathcal N)}^{N})$ of $\mathcal S_{\varphi,\max(\mathcal N)}$, orthonormal for the empirical inner product $\langle .,.\rangle_{N,1}$, obtained from $(\varphi_1,\dots,\varphi_{\max(\mathcal N)})$ via the Gram-Schmidt process. Then, $(\varphi_{j}^{N}\otimes\psi_{\ell})_{j,\ell}$ is an orthonormal basis of $\mathcal S_{\bf M}$ equipped with $\langle .,.\rangle_N$, so that
\begin{eqnarray*}
 \widehat\Pi_{\bf M}(p_t) & = &
 \sum_{j = 1}^{\max(\mathcal N)}\sum_{\ell = 1}^{\max(\mathcal N)}
 \langle p_t,\varphi_{j}^{N}\otimes\psi_{\ell}\rangle_N(\varphi_{j}^{N}\otimes\psi_{\ell})\\
 & &
 \hspace{2cm}{\rm and}\quad
 \widehat p_{\mathbf M,t} =
 \widehat\Pi_{\bf M}(\widehat p_{\mathbf M,t}) =
 \sum_{j = 1}^{\max(\mathcal N)}\sum_{\ell = 1}^{\max(\mathcal N)}
 \langle\widehat p_{\mathbf M,t},\varphi_{j}^{N}\otimes\psi_{\ell}\rangle_N(\varphi_{j}^{N}\otimes\psi_{\ell}).
\end{eqnarray*}
Thus,
\begin{eqnarray}
 |R_{\bf M}| & = &
 2\left|\left\langle
 \widehat\Pi_{\widehat{\bf m}}(p_t) -\widehat\Pi_{\bf m}(p_t),
 \sum_{j = 1}^{m_1\vee\widehat m_1}
 \sum_{\ell = 1}^{m_2\vee\widehat m_2}
 \langle\widehat p_{\mathbf M,t} - p_t,
 \varphi_{j}^{N}\otimes\psi_{\ell}\rangle_N(\varphi_{j}^{N}\otimes\psi_{\ell})
 \right\rangle_N\right|
 \nonumber\\
 & \leqslant &
 \frac{1}{4}\|\widehat\Pi_{\widehat{\bf m}}(p_t) -\widehat\Pi_{\bf m}(p_t)\|_{N}^{2} +
 4\left\|\sum_{j = 1}^{m_1\vee\widehat m_1}
 \sum_{\ell = 1}^{m_2\vee\widehat m_2}
 \langle\widehat p_{\mathbf M,t} - p_t,
 \varphi_{j}^{N}\otimes\psi_{\ell}\rangle_N(\varphi_{j}^{N}\otimes\psi_{\ell})
 \right\|_{N}^{2}
 \nonumber\\
 & \leqslant &
 \frac{1}{2}\|\widehat\Pi_{\widehat{\bf m}}(p_t) - p_{I\times J,t}\|_{N}^{2} +
 \frac{1}{2}\|\widehat\Pi_{\bf m}(p_t) - p_{I\times J,t}\|_{N}^{2} +
 4\sum_{j = 1}^{\widehat m_1}\sum_{\ell = 1}^{\widehat m_2}
 \langle\widehat p_{\mathbf M,t} - p_t,\varphi_{j}^{N}\otimes\psi_{\ell}\rangle_{N}^{2}
 \nonumber\\
 & &
 \hspace{8cm} +
 4\sum_{j = 1}^{m_1}\sum_{\ell = 1}^{m_2}
 \langle\widehat p_{\mathbf M,t} - p_t,\varphi_{j}^{N}\otimes\psi_{\ell}\rangle_{N}^{2}
 \nonumber\\
 & = &
 \frac{1}{2}\|\widehat\Pi_{\widehat{\bf m}}(p_t) - p_{I\times J,t}\|_{N}^{2} +
 \frac{1}{2}\|\widehat\Pi_{\bf m}(p_t) - p_{I\times J,t}\|_{N}^{2}
 \nonumber\\
 & &
 \hspace{5cm} +
 4\|\widehat\Pi_{\widehat{\bf m}}(\widehat p_{{\bf M},t} - p_t)\|_{N}^{2} +
 4\|\widehat\Pi_{\bf m}(\widehat p_{{\bf M},t} - p_t)\|_{N}^{2}
 \nonumber\\
 \label{risk_bound_adaptive_estimator_3}
 & = &
 \frac{1}{2}\|\widehat p_{\widehat{\bf m},t} - p_{I\times J,t}\|_{N}^{2} +
 \frac{1}{2}\|\widehat p_{\mathbf m,t} - p_{I\times J,t}\|_{N}^{2} +
 \frac{7}{2}\|\widehat p_{\widehat{\bf m},t} -\widehat\Pi_{\widehat{\bf m}}(p_t)\|_{N}^{2} +
 \frac{7}{2}\|\widehat p_{\mathbf m,t} -\widehat\Pi_{\bf m}(p_t)\|_{N}^{2}.
\end{eqnarray}
Therefore, by Inequalities (\ref{risk_bound_adaptive_estimator_2}) and (\ref{risk_bound_adaptive_estimator_3}),
\begin{eqnarray*}
 \|\widehat p_{\widehat{\bf m},t} - p_{I\times J,t}\|_{N}^{2}
 & \leqslant &
 3\|\widehat p_{\mathbf m,t} - p_{I\times J,t}\|_{N}^{2} +
 3\|\widehat p_{\mathbf m,t} -\widehat\Pi_{\bf m}(p_t)\|_{N}^{2}\\
 & &
 \hspace{3cm} +
 4\kappa{\rm pen}(\mathbf m) + 11\left(
 \|\widehat p_{\widehat{\bf m},t} -\widehat{\Pi}_{\widehat{\bf m}}(p_t)\|_{N}^{2} -
 \frac{2}{11}\kappa{\rm pen}(\widehat{\bf m})\right)\\
 & \leqslant &
 6\|\widehat p_{\mathbf m,t} - p_{I\times J,t}\|_{N}^{2} +
 4\kappa{\rm pen}(\mathbf m) + 11\left(
 \|\widehat p_{\widehat{\bf m},t} -\widehat{\Pi}_{\widehat{\bf m}}(p_t)\|_{N}^{2} -
 \frac{2}{11}\kappa{\rm pen}(\widehat{\bf m})\right)_+.
\end{eqnarray*}
{\bf Step 2.} First of all, consider
\begin{displaymath}
\Omega_N :=\bigcap_{m_1\in\mathfrak V_N}\Omega_{m_1}.
\end{displaymath}
For every $m_1\in\mathfrak V_N$,
\begin{displaymath}
\mathfrak L_{\varphi}(m_1)(\|\Psi_{m_1}\|_{\rm op}\vee 1)
\leqslant
\mathfrak c_{\varphi}^{2}m_1(\|\Psi_{m_1}\|_{\rm op}^{2}\vee 1)
\leqslant
4\mathfrak d\frac{NT}{\log(NT)}
\leqslant
\frac{\mathfrak c_{\Lambda}}{2}\cdot\frac{NT}{\log(NT)},
\end{displaymath}
leading to
\begin{equation}\label{risk_bound_adaptive_estimator_4}
\mathbb P(\Omega_{N}^{c})
\leqslant
\sum_{m_1\in\mathfrak V_N}\mathbb P(\Omega_{m_1}^{c})
\leqslant\frac{\mathfrak c_{\ref{Omega_event}}}{N^{p - 1}}
\quad\textrm{by Lemma \ref{Omega_event} (requiring Assumption \ref{condition_m}).}
\end{equation}
Now, since $\Xi_N =\{\mathcal M_N\subset\widehat{\mathcal M}_N\subset\mathfrak M_N\}$, on the event $\Xi_N\cap\Omega_N$, Inequality (\ref{risk_bound_adaptive_estimator_1}) remains true for every $\mathbf m\in\mathcal M_N$. Then,
\begin{eqnarray}
 \|\widehat p_{\widehat{\bf m},t} - p_{I\times J,t}\|_{N}^{2}
 & = &
 \|\widehat p_{\widehat{\bf m},t} - p_{I\times J,t}\|_{N}^{2}\mathbf 1_{\Xi_N\cap\Omega_N} +
 \|\widehat p_{\widehat{\bf m},t} - p_{I\times J,t}\|_{N}^{2}\mathbf 1_{\Xi_{N}^{c}\cup\Omega_{N}^{c}}
 \nonumber\\
 & \leqslant &
 \min_{\mathbf m\in\mathcal M_N}\{6\|\widehat p_{\mathbf m,t} - p_{I\times J,t}\|_{N}^{2} +
 4\kappa{\rm pen}(\mathbf m)\}
 \nonumber\\
 & &
 \hspace{1.5cm} +
 11\left(
 \|\widehat p_{\widehat{\bf m},t} -\widehat{\Pi}_{\widehat{\bf m}}(p_t)\|_{N}^{2} -
 \frac{2}{11}\kappa{\rm pen}(\widehat{\bf m})\right)_+\mathbf 1_{\Xi_N\cap\Omega_N}
 \nonumber\\
 & &
 \hspace{3cm} +
 \left(\|\widehat p_{\widehat{\bf m},t} -\widehat\Pi_{\widehat{\bf m}}(p_t)\|_{N}^{2} +
 \min_{\tau\in\mathcal S_{\widehat{\bf m}}}\|\tau - p_{I\times J,t}\|_{N}^{2}\right)
 (\mathbf 1_{\Xi_{N}^{c}} +\mathbf 1_{\Omega_{N}^{c}})
 \nonumber\\
 \label{risk_bound_adaptive_estimator_5}
 & \leqslant &
 \min_{\mathbf m\in\mathcal M_N}\{6\|\widehat p_{\mathbf m,t} - p_{I\times J,t}\|_{N}^{2} +
 4\kappa{\rm pen}(\mathbf m)\} + A + B
\end{eqnarray}
where, by Lemma \ref{empirical_risk_bound_key_lemma},
\begin{eqnarray*}
 A & := &
 \left(
 \sup_{\tau\in\mathcal S_{\widehat{\bf m}} :\|\tau\|_N = 1}\nu_N(\tau)^2 +\|p_{I\times J,t}\|_{N}^{2}\right)
 (\mathbf 1_{\Xi_{N}^{c}} +\mathbf 1_{\Omega_{N}^{c}})\\
 & &
 \hspace{2cm}{\rm and}\quad
 B :=
 11\left(
 \sup_{\tau\in\mathcal S_{\widehat{\bf m}} :\|\tau\|_N = 1}\nu_N(\tau)^2 -
 \frac{2}{11}\kappa{\rm pen}(\widehat{\bf m})\right)_+\mathbf 1_{\Xi_N\cap\Omega_N}.
\end{eqnarray*}
Let us find suitable bounds on $\mathbb E(A)$ and $\mathbb E(B)$. On the one hand, since
\begin{displaymath}
\sup_{\tau\in\mathcal S_{\bf m} :\|\tau\|_N = 1}\nu_N(\tau)^2
\leqslant 4m_1\mathfrak L_{\psi}(m_2)
\textrm{ $;$ }\forall\mathbf m = (m_1,m_2)\in\mathcal N^2
\end{displaymath}
as established in the proof of Theorem \ref{empirical_risk_bound} (see Inequality (\ref{empirical_risk_bound_2})), since $\mathbb E(\|p_{I\times J,t}\|_{N}^{4})^{1/2}\leqslant R(t)$ as established in the proof of Theorem \ref{f_weighted_risk_bound_truncated} (see Step 1), and by Inequality (\ref{risk_bound_adaptive_estimator_4}) and Lemma \ref{Xi_event},
\begin{displaymath}
\mathbb E(A)\leqslant
\mathfrak c_1\left(R(t) +\sum_{m_1,m_2 = 1}^{\max(\mathcal N)}m_1\mathfrak L_{\psi}(m_2)\right)
(\mathbb P(\Omega_{N}^{c})^{\frac{1}{2}} +\mathbb P(\Xi_{N}^{c})^{\frac{1}{2}})
\leqslant\frac{\mathfrak c_2(1 + R(t))}{N},
\end{displaymath}
where $\mathfrak c_1$ and $\mathfrak c_2$ are positive constants not depending on $N$ and $t$. On the other hand, since $\widehat{\bf m}\in\mathfrak M_N$ on the event $\Xi_N\cap\Omega_N$,
\begin{displaymath}
\mathbb E(B)
\leqslant
11\sum_{\mathbf m\in\mathfrak M_N}
\mathbb E\left[\left(\sup_{\tau\in\mathcal S_{\bf m} :\|\tau\|_N = 1}\nu_N(\tau)^2
-\frac{2}{11}\kappa{\rm pen}(\mathbf m)\right)_+\mathbf 1_{\Xi_N\cap\Omega_N}\right].
\end{displaymath}
Moreover, for every $\mathbf m = (m_1,m_2)\in\mathfrak M_N$,
\begin{displaymath}
\{\tau\in\mathcal S_{\bf m} :\|\tau\|_N = 1\}
\subset\{\tau\in\mathcal S_{\bf m} :\|\tau\|_{f}^{2}\leqslant 2\}
\quad {\rm on}\quad
\Omega_{m_1}\supset\Omega_N.
\end{displaymath}
Then,
\begin{displaymath}
\mathbb E(B)
\leqslant
11\sum_{\mathbf m\in\mathfrak M_N}\underbrace{
\mathbb E\left[\left(\sup_{\tau\in\mathcal S_{\bf m} :\|\tau\|_f\leqslant 2}\nu_N(\tau)^2
-\frac{2}{11}\kappa{\rm pen}(\mathbf m)\right)_+\mathbf 1_{\Xi_N\cap\Omega_N}\right]}_{=: b_{\bf m}}.
\end{displaymath}
 {\bf Step 3.} To prepare for the application of the integrated version (see Comte \cite{COMTE14}, Theorem A.1) of the Klein and Rio extension of Talagrand's inequality (see Klein and Rio \cite{KR05}, Theorem 1.2), this step deals with a suitable control of $b_{\bf m}$, $\mathbf m = (m_1,m_2)\in\mathfrak M_N$. First, for every $\tau =\sum_{j,\ell}T_{j,\ell}(\varphi_j\otimes\psi_{\ell})$ belonging to $\mathcal F :=\{\tau\in\mathcal S_{\bf m} :\|\tau\|_{f}^{2}\leqslant 2\}$,
\begin{displaymath}
\|\tau\|^2 = {\rm trace}(TT^*\Psi_{m_1}\Psi_{m_1}^{-1})
\leqslant\|\Psi_{m_1}^{-1}\|_{\rm op}
\underbrace{{\rm trace}(T^*\Psi_{m_1}T)}_{=\|\tau\|_{f}^{2}}\leqslant
2\|\Psi_{m_1}^{-1}\|_{\rm op},
\end{displaymath}
and then
\begin{eqnarray*}
 \|\tau\|_{\infty}
 & = &
 \sup_{(x,y)\in I\times J}\left\{
 \sum_{j = 1}^{m_1}\sum_{\ell = 1}^{m_2}
 T_{j,\ell}\varphi_j(x)\psi_{\ell}(y)\right\}\\
 & \leqslant &
 \mathfrak L_{\psi}(m_2)^{\frac{1}{2}}
 \sup_{x\in I}\left\{ 
 \sum_{j = 1}^{m_1}|\varphi_j(x)|
 \sum_{\ell = 1}^{m_2}|T_{j,\ell}|\right\}\leqslant
 \sqrt 2\mathfrak L_{\varphi}(m_1)^{\frac{1}{2}}
 \mathfrak L_{\psi}(m_2)^{\frac{1}{2}}\|\Psi_{m_1}^{-1}\|_{\rm op}^{\frac{1}{2}}.
\end{eqnarray*}
Since $m_1\in\mathfrak V_N$, and since $\mathfrak L_{\psi}(m_2)\leqslant m_1\mathfrak L_{\psi}(m_2)\leqslant N$ by the definition of $\mathcal U_N$,
\begin{displaymath}
\sup_{\tau\in\mathcal F}\|\tau\|_{\infty}
\leqslant M
\quad {\rm with}\quad
M =\frac{\mathfrak c_3N}{\log(N)^{1/2}}
\quad {\rm and}\quad
\mathfrak c_3 = 2\sqrt{2\mathfrak dT}.
\end{displaymath}
Now, since
\begin{displaymath}
\mathbb E\left(\sup_{\tau\in\mathcal F}\nu_N(\tau)^2\right)\leqslant
\frac{2m_1\mathfrak L_{\psi}(m_2)}{N}
\end{displaymath}
as established in the proof of Theorem \ref{empirical_risk_bound} (see Step 1),
\begin{displaymath}
\mathbb E\left(\sup_{\tau\in\mathcal F}|\nu_N(\tau)|\right)\leqslant H
\quad {\rm with}\quad
H^2 =\frac{2m_1\mathfrak L_{\psi}(m_2)}{N},
\end{displaymath}
and then
\begin{displaymath}
N\sup_{\tau\in\mathcal F}\mathbb E(\nu_N(\tau)^2)\leqslant\upsilon
\quad {\rm with}\quad
\upsilon = NH^2.
\end{displaymath}
By the aforementioned Talagrand's inequality, for $\alpha := a\log(N)$,
\begin{displaymath}
\mathbb E\left[\left(\sup_{\tau\in\mathcal F}\nu_N(\tau)^2 -
2(1 + 2\alpha)H^2\right)_+\right]\leqslant
A(N,\alpha,H,\upsilon) + B(N,\alpha,M,H),
\end{displaymath}
where
\begin{eqnarray*}
 A(N,\alpha,H,\upsilon) & := &
 \frac{24\upsilon}{N}\exp\left(-\frac{\alpha NH^2}{6\upsilon}\right)\\
 & = &
 \frac{48m_1\mathfrak L_{\psi}(m_2)}{N}
 e^{-\frac{a}{6}\log(N)}
 \leqslant 48N^{-\frac{a}{6}}
 \quad\textrm{because $\mathbf m\in\mathcal U_N$},
\end{eqnarray*}
and
\begin{eqnarray*}
 B(N,\alpha,M,H) & := &
 \frac{\mathfrak c_4M^2}{N^2}\exp\left(-\frac{\sqrt{2}}{42}\cdot\frac{NH\sqrt\alpha}{M}\right)
 \quad {\rm with}\quad
 \mathfrak c_4 = 7056\\
 & = &
 \frac{\mathfrak c_{3}^{2}\mathfrak c_4}{\log(N)}\exp\left(-\frac{\sqrt{2}}{42}\cdot
 \frac{\sqrt{2a}\log(N)}{\mathfrak c_3}\right)
 \leqslant
 \mathfrak c_{3}^{2}\mathfrak c_4
 N^{-\frac{\sqrt{2a}}{84\sqrt{\mathfrak dT}}}.
\end{eqnarray*}
Therefore, since $a\geqslant (2\cdot 84\sqrt{\mathfrak dT})^2/2$, and since
\begin{displaymath}
\kappa\geqslant\kappa_0 = 44a\geqslant
\frac{22(1 + 2a\log(N))}{1 +\log(N)} =
11(1 + 2a\log(N))\frac{H^2}{{\rm pen}(\mathbf m)},
\end{displaymath}
there exists a constant $\mathfrak c_5 > 0$, not depending on $N$, $m_1$, $m_2$ and $t$, such that
\begin{eqnarray*}
 b_{\bf m} & \leqslant &
 \mathbb E\left[\left(\sup_{\tau\in\mathcal S_{\bf m} :\|\tau\|_f\leqslant 2}\nu_N(\tau)^2
 -\frac{2}{11}\kappa{\rm pen}(\mathbf m)\right)_+\right]\\
 &\leqslant &
 \mathbb E\left[\left(\sup_{\tau\in\mathcal F}\nu_N(\tau)^2 -
 2(1 + 2\alpha)H^2\right)_+\right]\leqslant\frac{\mathfrak c_5}{N^2}.
\end{eqnarray*}
{\bf Step 4 (conclusion).} By Steps 2 and 3,
\begin{displaymath}
\mathbb E(A)\leqslant
\frac{\mathfrak c_2(1 + R(t))}{N}
\quad {\rm and}\quad
\mathbb E(B)\leqslant
11\sum_{\mathbf m\in\mathfrak M_N}b_{\bf m}
\leqslant
\frac{11\mathfrak c_5}{N}.
\end{displaymath}
Thus, by Inequality (\ref{risk_bound_adaptive_estimator_5}), there exists a constant $\mathfrak c_6 > 0$, not depending on $N$ and $t$, such that
\begin{displaymath}
\mathbb E(\|\widehat p_{\widehat{\bf m},t} - p_{I\times J,t}\|_{N}^{2})
\leqslant
6\min_{\mathbf m\in\mathcal M_N}\{\mathbb E(\|\widehat p_{\mathbf m,t} - p_{I\times J,t}\|_{N}^{2}) +
\kappa{\rm pen}(\mathbf m)\} +
\frac{\mathfrak c_6(1 + R(t))}{N}.
\end{displaymath}
Moreover, Theorem \ref{risk_bound_adaptive_estimator}.(2) is a consequence of Theorem \ref{risk_bound_adaptive_estimator}.(1) thanks to Inequality (\ref{Gaussian_bound_density}) as in the last step of the proof of Theorem \ref{section_proof_f_weighted_risk_bound_truncated}.
%

% Subsubsection : Proof of Lemma projection_estimator_smaller_model.

%
\subsubsection{Proof of Lemma \ref{projection_estimator_smaller_model}}\label{section_proof_projection_estimator_smaller_model}
By the definition of $\widehat\Pi_{\bf m}$ (see (\ref{empirical_orthogonal_projection})),
\begin{displaymath}
\widehat\Pi_{\bf m}(\widehat p_{\mathbf M,t}) =
\sum_{j = 1}^{m_1}\sum_{\ell = 1}^{m_2}
[\widehat\Psi_{m_1}^{-1}\widehat P_{\bf m}(
\widehat p_{\mathbf M,t})]_{j,\ell}(\varphi_j\otimes\psi_{\ell})
\end{displaymath}
with, for every $j\in\{1,\dots,m_1\}$ and $\ell\in\{1,\dots,m_2\}$,
\begin{eqnarray*}
 \widehat P_{\bf m}(\widehat p_{\mathbf M,t})
 & = &
 \frac{1}{NT}\sum_{i = 1}^{N}\int_{0}^{T}\int_{-\infty}^{\infty}
 \widehat p_{\mathbf M,t}(X_{s}^{i},y)\varphi_j(X_{s}^{i})\psi_{\ell}(y)dyds\\
 & = &
 \sum_{j' = 1}^{M_1}\sum_{\ell' = 1}^{M_2}
 [\widehat\Psi_{M_1}^{-1}\widehat Z_{\mathbf M,t}]_{j',\ell'}\\
 & &
 \hspace{2cm}\times
 \frac{1}{NT}\sum_{i = 1}^{N}\int_{0}^{T}\varphi_j(X_{s}^{i})\varphi_{j'}(X_{s}^{i})
 \underbrace{\int_{-\infty}^{\infty}\psi_{\ell}(y)\psi_{\ell'}(y)dy}_{=\delta_{\ell,\ell'}}ds\\
 & = &
 \sum_{j' = 1}^{M_1}
 [\widehat\Psi_{M_1}^{-1}\widehat Z_{\mathbf M,t}]_{j',\ell}\underbrace{
 \frac{1}{NT}\sum_{i = 1}^{N}\int_{0}^{T}
 \varphi_j(X_{s}^{i})\varphi_{j'}(X_{s}^{i})ds}_{= [\widehat\Psi_{M_1}]_{j,j'}}\\
 & = &
 [\widehat Z_{\mathbf M,t}]_{j,\ell} = [\widehat Z_{\mathbf m,t}]_{j,\ell}
 \quad {\rm because}\quad
 \mathcal S_{\bf m}\subset\mathcal S_{\bf M}.
\end{eqnarray*}
Therefore,
\begin{displaymath}
\widehat\Pi_{\bf m}(\widehat p_{\mathbf M,t}) =
\sum_{j = 1}^{m_1}\sum_{\ell = 1}^{m_2}
[\widehat\Psi_{m_1}^{-1}\widehat Z_{\mathbf m,t}]_{j,\ell}(\varphi_j\otimes\psi_{\ell}) =
\widehat p_{\mathbf m,t}.
\end{displaymath}
%

% Subsection : Proof of Proposition risk_bound_adaptive_estimator_bis.

%
\subsection{Proof of Proposition \ref{risk_bound_adaptive_estimator_bis}}
The proof is mainly the same as that of Theorem \ref{risk_bound_adaptive_estimator}, except for the control $M$ of $\sup_{\tau\in\mathcal F}\|\tau\|_{\infty}$ and the control $\upsilon$ of $N\sup_{\tau\in\mathcal F}\mathbb E(\nu_N(\tau)^2)$ involved in Talagrand's inequality (see Step 3 in Subsubsection \ref{section_steps_proof_theorem_risk_bound_adaptive_estimator}).
\\
\\
{\bf Step 3 (bis).} Thanks to the Klein and Rio version of Talagrand's inequality (see Klein and Rio \cite{KR05}), this step deals with a suitable control of $b_{\bf m}$, $\mathbf m = (m_1,m_2)\in\mathfrak M_N$. First, for every $\tau =\sum_{j,\ell}T_{j,\ell}(\varphi_j\otimes\psi_{\ell})$ belonging to $\mathcal F =\{\tau\in\mathcal S_{\bf m} :\|\tau\|_{f}^{2}\leqslant 2\}$, we still get
\begin{eqnarray*}
 \|\tau\|_{\infty}
 & \leqslant &
 \sqrt 2\mathfrak L_{\varphi}(m_1)^{\frac{1}{2}}
 \mathfrak L_{\psi}(m_2)^{\frac{1}{2}}\|\Psi_{m_1}^{-1}\|_{\rm op}^{\frac{1}{2}}.
\end{eqnarray*}
By Assumption \ref{assumption_projection_spaces}.(3), and since $\mathbf m = (m_1,m_2)$ belongs to $\mathfrak M_N$,
\begin{displaymath}
\mathfrak L_{\varphi}(m_1)^{\frac{1}{2}}
\|\Psi_{m_1}^{-1}\|_{\rm op}^{\frac{1}{2}}\leqslant
\mathfrak c_{\varphi}^{\frac{1}{2}}m_{1}^{\frac{1}{4}}
(\mathfrak c_{\varphi}^{2}m_1\|\Psi_{m_1}^{-1}\|_{\rm op}^{2})^{\frac{1}{4}}\leqslant
\mathfrak c_{\varphi}^{\frac{1}{2}}m_{1}^{\frac{1}{4}}
\left(4\mathfrak d \frac{NT}{\log(NT)}\right)^{\frac{1}{4}}.
\end{displaymath}
Thus, for $NT\geqslant e$,
\begin{displaymath}
\sup_{\tau\in\mathcal F}\|\tau\|_{\infty}
\leqslant M
\quad {\rm with}\quad
M =\mathfrak c_{3,b}m_{1}^{\frac{1}{4}}\mathfrak L_{\psi}(m_2)^{\frac{1}{2}}N^{\frac{1}{4}} 
\quad {\rm and}\quad
\mathfrak c_{3,b} = 2\sqrt{\mathfrak c_{\varphi}}(\mathfrak dT)^{\frac{1}{4}}.
\end{displaymath}
Now, since
\begin{displaymath}
\mathbb E\left(\sup_{\tau\in\mathcal F}\nu_N(\tau)^2\right)\leqslant
\frac{2m_1\mathfrak L_{\psi}(m_2)}{N}
\end{displaymath}
as established in the proof of Theorem \ref{empirical_risk_bound} (see Step 1),
\begin{displaymath}
\mathbb E\left(\sup_{\tau\in\mathcal F}|\nu_N(\tau)|\right)\leqslant H
\quad {\rm with}\quad
H^2 =\frac{2m_1\mathfrak L_{\psi}(m_2)}{N}.
\end{displaymath}
Lastly,
\begin{eqnarray*}
 & &
 \sup_{\tau\in\mathcal F}\left\{
 {\rm var}\left(\frac{1}{T}\int_{0}^{T}(
 \tau(X_s,X_{s + t}) -\mathbb E(\tau(X_s,X_{s + t})|X_s))ds\right)\right\}\\
 & &
 \hspace{4cm}
 \leqslant\sup_{\tau\in\mathcal F}
 \mathbb E\left[\left(\frac{1}{T}\int_{0}^{T}\tau(X_s,X_{s + t})ds\right)^2\right]
 \leqslant
 p_0\sup_{\tau\in\mathcal F}\|\tau\|_{f}^{2}
 \leqslant\upsilon
 \quad {\rm with }\quad\upsilon = 2p_0.
\end{eqnarray*}
By the aforementioned Talagrand's inequality, for $\alpha = 1/4$,
\begin{displaymath}
\mathbb E\left[\left(\sup_{\tau\in\mathcal F}\nu_N(\tau)^2 - 3H^2\right)_+\right]\leqslant
A(N,H,\upsilon) + B(N,M,H),
\end{displaymath}
where
\begin{displaymath}
A(N,H,\upsilon) :=
\frac{24\upsilon}{N}\exp\left(-\frac{NH^2}{24\upsilon}\right) = 
\frac{48p_0}{N}\exp\left(-\frac{m_1\mathfrak L_{\psi}(m_2)}{24p_0} \right),
\end{displaymath}
and
\begin{eqnarray*}
 B(N,M,H) & := &
 \frac{\mathfrak c_{4,b}M^2}{N^2}\exp\left(-\frac{\sqrt 2}{84}\cdot\frac{NH}{M}\right)
 \quad {\rm with}\quad
 \mathfrak c_{4,b} = 7056\\
 & \leqslant &
 \frac{\mathfrak c_{4,b}\mathfrak c_{3,b}^{2}}{\sqrt{N}}
 \exp\left(-\frac{m_{1}^{1/4}N^{1/4}}{42\mathfrak c_{3,b}}\right)
 \quad {\rm because}\quad\mathbf m\in\mathcal U_N.
\end{eqnarray*}
Since $m_1\geqslant 1$,
\begin{displaymath}
B(N,M,H)\leqslant
\frac{\mathfrak c_{5,b}}{\sqrt N}
\exp(-\mathfrak c_{6,b}N^{\frac{1}{4}})
\quad {\rm with}\quad
\mathfrak c_{5,b} =\mathfrak c_{4,b}\mathfrak c_{3,b}^{2}
\quad {\rm and}\quad
\mathfrak c_{6,b} =\frac{1}{42\mathfrak c_{3,b}}.
\end{displaymath}
So, for $\kappa_b\geqslant\kappa_{b,0} = 33/2$ and $t\in [t_0,T]$,
\begin{eqnarray*}
 b_{\bf m} & \leqslant &
 \mathbb E\left[\left(\sup_{\tau\in\mathcal S_{\bf m} :\|\tau\|_f\leqslant 2}\nu_N(\tau)^2 -
 \frac{2}{11}\kappa{\rm pen}(\mathbf m)\right)_+\right]
 \leqslant 
 \mathbb E\left[\left(\sup_{\tau\in\mathcal F}\nu_N(\tau)^2 - 3H^2\right)_+\right]\\
 &\leqslant &
 \frac{48p_0}{N}
 \exp\left(-\frac{m_1\mathfrak L_{\psi}(m_2)}{24p_0}\right) +
 \frac{\mathfrak c_{5,b}}{\sqrt N}\exp(-\mathfrak c_{6,b}N^{\frac{1}{4}}).
\end{eqnarray*}
Therefore, by Assumption \ref{condition_m_bis}, there exists a constant $\mathfrak c_{7,b} > 0$, not depending on $N$, such that
\begin{displaymath}
\mathbb E(B)
\leqslant 11\sum_{\mathbf m\in\mathfrak M_N}b_{\bf m}\leqslant
\frac{\mathfrak c_{7,b}}{N}.
\end{displaymath}
%

% References.

%

%
\end{document}